\documentclass[leqno,10pt]{amsart}

\usepackage{amsmath,amssymb,mathrsfs} 
\usepackage{times}
\usepackage{gensymb}
\usepackage[all]{xy}

\theoremstyle{plain}
    \newtheorem{thm}{Theorem}[section]
    \newtheorem{ppn}[thm]{Proposition}
    \newtheorem{lem}[thm]{Lemma}
    \newtheorem{cor}[thm]{Corollary}
\theoremstyle{definition}
    \newtheorem{dfn}[thm]{Definition}
    
\theoremstyle{remark}
    \newtheorem{rmk}[thm]{Remark}
    \newtheorem*{rmk*}{Remark}   
    \newtheorem{epl}[thm]{Example}

\numberwithin{equation}{section}

\allowdisplaybreaks

\def\Hom{\operatorname{Hom}}
\def\C{\mathbb{C}}\def\Q{\mathbb{Q}}\def\F{\mathbb{F}}\def\N{\mathbb{N}}
\def\ol#1{\overline{#1}}\def\wt#1{\widetilde{#1}}\def\wh#1{\widehat{#1}}

\def\ot{\otimes}
\def\a{\alpha}\def\b{\beta}\def\g{\gamma}\def\d{\delta}\def\e{\varepsilon}
\def\l{\lambda}\def\k{\kappa}\def\z{\zeta}\def\x{\xi}\def\y{\eta}

\def\G{\Gamma}

\def\L{\Lambda}\def\s{\sigma}\def\t{\tau}

\def\Re{\operatorname{Re}}

\def\FF#1#2#3{F\left({#1\atop#2};#3\right)}
\def\FFF#1#2#3{{}_3F_2\left({#1\atop #2};#3\right)}

\def\r{\rho}

\newcommand{\Gal}{\mathrm{Gal}}

\newcommand{\A}{{\mathbb{A}}}

\newcommand{\baa}{\boldsymbol\alpha}
\newcommand{\bbb}{\boldsymbol\beta}
\newcommand{\bgg}{\boldsymbol\gamma}

\newcommand\n{\nu}\newcommand\m{\mu}
\newcommand\vp{\varphi}

\newcommand\ck{\wh{\k^*}}
\newcommand\0{\circ}

\newcommand{\Ind}{\operatorname{Ind}}

\newcommand{\supp}{\operatorname{supp}}

\begin{document}

\title{Hypergeometric functions over finite fields}
\author{Noriyuki Otsubo}
\email{otsubo@math.s.chiba-u.ac.jp}
\address{Department of Mathematics and Informatics, Chiba University, Inage, Chiba, 263-8522 Japan}

\begin{abstract}
We give a definition of generalized hypergeometric functions over finite fields using modified Gauss sums, 
which enables us to find clear analogy with classical hypergeometric functions over the complex numbers. 
We study their fundamental properties and prove summation formulas, transformation formulas and product formulas.  
An application to zeta functions of K3-surfaces is given. 
In the appendix, we give an elementary proof of the Davenport--Hasse multiplication formula for Gauss sums. 
\end{abstract}

\date{\today.}
\subjclass[2010]{11T24, 11L05, 33C05, 33C20}
\keywords{Hypergeometric functions, Finite fields, Exponential sums, Zeta functions}

\maketitle

\section{Introduction}

Recall that the classical complex hypergeometric function ${}_rF_s(x)$ is defined by the power series
$${}_rF_s\left({a_1,\dots, a_r \atop b_1,\dots, b_s};x\right)
=\sum_{n=0}^\infty \frac{\prod_{i=1}^r (a_i)_n}{(1)_n \prod_{i=1}^s (b_i)_n} x^n, $$
where $a_i$, $b_i$  are complex parameters with $-b_i \not\in \N:=\{0,1,\dots\}$,   
and the Pochhammer symbol $(a)_n$ is defined by the gamma function as 
$$(a)_n=\frac{\G(a+n)}{\G(a)}.$$
Its special values have been of particular interest. For example, we have the classical Euler--Gauss summation formula
$${}_2F_1\left({a,b \atop c};1\right) = \frac{\G(c)\G(c-a-b)}{\G(c-a)\G(c-b)}$$
if $\Re(c-a-b)>0$. 

Over a finite field, an analogue of the Euler--Gauss formula was first studied by Helversen-Pasotto \cite{helversen-pasotto}. 
It seems that hypergeometric functions over a finite field first appeared in Koblitz's work \cite{koblitz}. 
In the same manner as Weil's work on Fermat varieties \cite{weil}, 
he computed the number of rational points on the variety defined by 
$$y^{n} = (1-\l x_1\cdots x_d)^{a_0} x_1^{a_1}(1-x_1)^{b_1}\cdots x_d^{a_d}(1-x_d)^{b_d},$$
which generalizes the Legendre elliptic curve. 
Decomposing the number of rational points by the action of the group of $n$th roots of unity, 
he arrived at a definition of ${}_{d+1}F_d(\l)$. 
Other definitions of hypergeometric functions over finite fields were given by Greene \cite{greene}, Katz \cite{katz}, McCarthy \cite{mccarthy} and Fuselier et. al. \cite{fuselieretal}, 
where, except for \cite{katz}, only the case when $r=s+1$ is considered 
(see Remark \ref{comparison}). 

In this paper, we give a definition of ${}_rF_s$-functions over a finite field $\k$, which coincides with McCarthy's definition when $r=s+1$. 
Recall that a finite analogue of the gamma function is Gauss sums, denoted by $g(\a)$, viewed as a $\C$-valued function on 
$\ck$, the set of multiplicative characters of $\k$. 
Therefore an analogue of the Pochhammer symbol is defined naturally by 
$$(\a)_\n=\frac{g(\a\n)}{g(\a)}$$ where $\a$, $\n\in\ck$. 
The novelty of our definition is simply an introduction of modified symbols $g^\0(\a)$ and $(\a)^\circ_\n$ used for the denominator parameters. We define 
$${}_rF_s\left({\a_1,\dots, \a_r \atop \b_1,\dots, \b_s};\l\right)
=\frac{1}{1-q}
\sum_{\n\in\ck} \frac{\prod_{i=1}^r (\a_i)_\n}{(\e)^\0_\n \prod_{i=1}^s (\b_i)^\0_\n} \n(\l), $$
where $\e$ is the trivial character and $q=\#\k$. 
One can define by the same principle more general hypergeometric functions with many variables (see Section 2.4). 

We will prove some finite analogues of results classically known for complex hypergeometric functions, 
such as summation formulas, transformation formulas and product formulas.   
Because of our definition, not only the statements but also some proofs become quite parallel to the complex case, 
although differential equations are not available. 
Moreover, the case where $r \ne s+1$ can be treated equally. 
Some of the results in this paper are already known, at least essentially or under different hypotheses, 
but we give self-contained proofs together with references to the literature 
(e.g. Evans--Greene \cite{evans2} \cite{evans}, Greene \cite{greene}, McCarthy \cite{mccarthy}). 

The strong similarities between complex and finite hypergeometric functions are not coincidental. 
Just as the relations between the gamma function and Gauss sums or between the beta function and Jacobi sums, 
complex and finite hypergeometric functions should be associated to different realizations of same ``motives", pure motives in the strict sense when $r=s+1$. 
This perspective will be further investigated elsewhere.

This paper is constructed as follows. 
In Section 2, after recalling basic facts about Gauss and Jacobi sums, we give our definitions. 
In Section 3, we prove reduction and iteration formulas which reduce a hypergeometric function to one with a smaller number of parameters. We also prove finite analogues of transformation formulas of Euler and Pfaff, and derive relations among finite analogues of Kummer's 24 functions.  
In section 4, we prove formulas on special values ${}_{d+1}F_d(\pm 1)$ analogous to classical formulas of Euler--Gauss, Kummer, Thomae, Dixon, Watson, Whipple, Saalsch\"utz, etc. 
In Section 5, we prove quadratic transformation formulas analogous to classical formulas of Gauss, Kummer, Ramanujan, etc. 
In Section 6, we prove some product formulas analogous to classical formulas of Kummer and Ramanujan concerning confluent hypergeometric functions ${}_1F_1$ and of Clausen. 
In the last Section 7, we give an application to zeta functions of K3 surfaces. 
As an appendix, we include an elementary proof of the Davenport--Hasse multiplication formula for Gauss sums.

\section{Definitions}

Throughout this paper, $\k$ denotes a finite field of characteristic $p$ with $q$ elements. 
Let $\ck=\Hom(\k^*,\C^*)$ denote the group of multiplicative characters of $\k$, and let 
$\e \in \wh{\k^*}$ denote the unit character. For any $\vp\in\ck$, we set $\vp(0)=0$ and write $\ol\vp=\vp^{-1}$. 
For a group $G$, let $\d\colon G \to \{0,1\}$ denote the characteristic function of the identity element. 
This notation will be applied to $G=\k$ (the additive group) and to $G=\ck$. 
For a positive integer $n$, let $\mu_n\subset \C^*$ denote the group of $n$th roots of unity. 

\subsection{Gauss and Jacobi sums}

\begin{dfn}
Fix a non-trivial additive character $\psi\in \Hom(\k,\C^*)$.  
For $\vp\in \wh{\k^*}$, define the {\em Gauss sum} and its variant by
$$g(\vp) =-\sum_{x\in\k}\psi(x)\vp(x), \quad  g^\circ(\vp)=q^{\d(\vp)}g(\vp)\quad \in \Q(\mu_{p(q-1)}).$$
For $\vp_1, \dots, \vp_n \in \wh{\k^*}$ ($n\ge 1$), define the {\em Jacobi sum} by 
$$j(\vp_1,\dots, \vp_n) = (-1)^{n-1} \sum_{x_1,\dots, x_n \in\k \atop x_1+\cdots+x_n =1}\vp_1(x_1)\cdots \vp_n(x_n) \quad \in \Q(\mu_{q-1}). $$
\end{dfn}

We recall the basic properties of Gauss and Jacobi sums. 

\begin{ppn}\label{jacobi}\  
\begin{enumerate}
\item We have $g(\e)=1$ and $g^\circ(\e)=q$. 
\item For any $\vp\in\ck$, 
$$\ol{g(\vp)}=\vp(-1)g(\ol\vp), \quad \ol{g^\0(\vp)}=\vp(-1)g^\0(\ol\vp).$$ 
\item For any $\vp\in\ck$, 
$$g(\vp)g^\circ(\ol\vp)=\vp(-1)q.$$ 
In particular,  
$$|g(\vp)|={\sqrt q}^{1-\d(\vp)}, \quad 
|g^\0(\vp)|={\sqrt q}^{1+\d(\vp)}.$$  
\item For any $\vp_1,\dots, \vp_n \in \ck$ ($n \ge 1$), 
$$j(\vp_1, \dots,\vp_n)=
\begin{cases}
\dfrac{g(\vp_1) \cdots g(\vp_n)}{g^\0(\vp_1\cdots \vp_n)} & 
\text{if $(\vp_1, \dots, \vp_n) \ne(\e,\dots, \e)$},  \\
\dfrac{1-(1-q)^n}{q} & \text{if $(\vp_1, \dots, \vp_n)=(\e,\dots, \e)$}.
\end{cases}$$
\end{enumerate}
\end{ppn}

\begin{proof}
These are standard, but for the convenience of the reader, we give a short proof of (iv) (the others are easier). 
For any $a \in \k$, put
$$j_a=(-1)^{n-1}\sum_{x_1+\cdots+x_n=a}\vp_1(x_1)\cdots \vp_n(x_n),$$
so that $j_1=j(\vp_1,\dots, \vp_n)$.  
If we put $\vp_0=\ol{\vp_1\cdots \vp_n}$, then $j_a=\ol{\vp_0}(a) j_1$ for any $a \ne 0$. 
We have 
$$\sum_{a \in \k} j_a= (-1)^{n-1}\prod_{i=1}^n \sum_{x_i\in \k } \vp_i(x_i)=
\begin{cases}
0 &  \text{if $(\vp_1, \dots, \vp_n)\ne(\e,\dots, \e)$}, 
\\ -(1-q)^n &  \text{if $(\vp_1, \dots, \vp_n)=(\e,\dots, \e)$}.
\end{cases}$$ 
On the other hand, 
\begin{align*}
g(\vp_0)g(\vp_1)\cdots g(\vp_n) & = (-1)^{n+1} \sum_{z\in\k} \psi(z) \sum_{x_0+\cdots+x_n=z} \vp_0(x_0)\cdots \vp_n(x_n)
\\&=\sum_{z\in\k} \psi(z) \sum_{x_0\ne 0}
(-1)^{n-1} \sum_{x_1'+\cdots +x_n'=\frac{z}{x_0}-1} \vp_1(x_1')\cdots \vp_n(x_n'), 
\end{align*}
and the second sum in the last member is $(q-1)j_{-1}$ if $z=0$, and  $\sum_{a \ne -1} j_a$ if $z \ne 0$. 
Hence it follows
$$g(\vp_0)g(\vp_1)\cdots g(\vp_n)=
\begin{cases}
qj_{-1} & \text{if $(\vp_1, \dots, \vp_n) \ne(\e,\dots, \e)$}, \\ 
qj_{-1} +(1-q)^n & \text{if $(\vp_1, \dots, \vp_n) =(\e,\dots, \e)$}, 
\end{cases}$$ 
and the formula follows using (i) and (iii). 
\end{proof}

\begin{rmk}
The Gauss sums $g(\vp)$ and $g^\0(\vp)$ can be interpreted geometrically as follows. 
Let $C$ be the affine Artin--Schreier curve over $\k$ defined by 
$$y^{q-1}=x^q-x, \quad y \ne 0,$$
on which the group $\k \times \k^*$ acts by addition in $x$ and multiplication in $y$. 
Then the $g(\vp)$ (resp. $g^\0(\vp)$) is the trace of Frobenius acting on the (one-dimensional) $(\psi,\vp)$-eigenspace in the first $\ell$-adic cohomology group of $C$ with (resp. without) compact support for any prime $\ell \nmid p(q-1)$. 
Proposition \ref{jacobi} (iii) results from the Poincar\'e duality, i.e. the Galois equivariant perfect pairing 
$$H^1_c(C)^{(\psi,\vp)} \ot H^1(C)^{(\ol\psi,\ol\vp)} \to\ol{\Q_\ell}(-1).$$ 
\end{rmk}

\subsection{Pochhammer symbols}

Recall that the complex Pochhammer symbol is defined for $a\in\C$ and $n\in\N$ by 
$$(a)_n=\frac{\G(a+n)}{\G(a)}=a(a+1)\cdots (a+n-1).$$

\begin{dfn} 
For any $\a$, $\n \in \wh{\k^*}$, define the {\em Pochhammer symbol} and its variant by 
$$(\a)_\n=\frac{g(\a\n)}{g(\a)}, \quad 
(\a)^\circ_\n=\frac{g^\circ(\a\n)}{g^\circ(\a)} \quad \in \Q(\mu_{p(q-1)}).$$
For example, 
$$(\a)_\e=(\a)^\circ_\e=1, \quad (\e)_\n=g(\n), \quad (\e)^\circ_\n=g^\circ(\n)/q.$$ 
\end{dfn}

\begin{lem}\label{l2.1}\ 
\begin{enumerate}
\item For any $\a$, $\b$, $\n \in \wh{\k^*}$, 
$$(\a)_{\b\n}=(\a)_\b(\a\b)_{\n}, \quad 
(\a)^\circ_{\b\n}=(\a)^\circ_\b(\a\b)^\circ_{\n}.$$
\item
For any $\a$, $\n \in \wh{\k^*}$, 
$$(\a)_\n(\ol\a)^\0_{\ol\n}=\n(-1).$$ 
\item
For any $\a, \b, \n \in \wh{\k^*}$, ${(\a)_\n}/{(\b)^\circ_\n}$
is an element of $\Q(\mu_{q-1})$ independent of the choice of $\psi$.  
\end{enumerate}
\end{lem}
\begin{proof}
The statement (i) is evident and  (ii) follows from Proposition \ref{jacobi} (iii). 
The statement (iii) is evident if $\a=\b$. Otherwise, this follows since 
$${(\a)_\n}/{(\b)^\circ_\n}={j(\a\n,\ol\a\b)}/{j(\a,\ol\a\b)}$$
by Proposition \ref{jacobi} (iv).  
\end{proof}

\begin{dfn}
Define the set of parameters by
$$P=\left\{\sum_{\vp\in\ck} n_\vp \vp \Bigm| n_\vp \in \N\right\},$$
the free abelian monoid over $\ck$. Let 
$$\deg \colon P \to \N; \quad \sum_\vp n_\vp \vp \mapsto \sum_\vp n_\vp$$
be the degree map, and put $P_d=\deg^{-1}(d)$, so that $P_1=\ck$.  Let 
$$(\ , \ ) \colon P \times P \to \N$$
be the symmetric bi-additive map extending $(\a,\vp) \mapsto \d(\a\ol\vp)$, so that 
$\baa= \sum_{\vp} (\baa,\vp)\vp$ for any $\baa \in P$. 
Extend the Pochhammer symbols $(\a)_\n$, $(\a)_\n^\0$ to monoid homomorphisms $P \to \C^*$, i.e.   
$$(\baa)_\n=\prod_\vp (\vp)_\n^{(\baa,\vp)}, \quad (\baa)_\n^\0=\prod_\vp {(\vp)^\0_\n}^{(\baa,\vp)}$$
for any $\baa\in P$ and $\n\in\ck$. 
Then we have 
$$\frac{(\baa)_\n}{(\baa)^\0_\n}=q^{(\baa,\e)-(\baa,\ol\n)}.$$
\end{dfn}

\subsection{Fourier transform}\label{s.fourier}

For a function $f\colon \k^* \to \C$, 
its {\em Fourier transform} is a function $\wh f \colon \ck\to \C$ defined by 
$$\wh f(\n) = \sum_{\l\in\k^*} f(\l)\ol\n(\l).$$
We have the Fourier inversion theorem 
$$f(\l)=\frac{1}{q-1}\sum_{\n\in\ck}\wh f(\n) \n(\l).$$
Conversely if $f(\l)=(q-1)^{-1}\sum_{\n\in\ck} a_\n \n(\l)$, then 
$\wh f(\n)=a_\n$. 

We have also the convolution formula 
$$\wh{f_1f_2}(\n)= \frac{1}{q-1} \sum_{\n_1\n_2=\n} \wh{f_1}(\n_1)\wh{f_2}(\n_2), $$
and the Plancherel formula 
$$\sum_{\l\in\k^*} f_1(\l)\ol{f_2(\l)} = \frac{1}{q-1} \sum_{\n\in\ck} \wh f_1(\n) \ol{\wh f_2(\n)}.$$

\begin{epl}\label{e2.2}\ 
\begin{enumerate}
\item
For $a\in\k^*$, 
$f(\l)=\d(a-\l)$ if and only if $\wh f(\n)=\ol\n(a)$.  

\item For the additive character $\psi$, $f(\l)=\psi(\l)$ if and only if 
$\wh f(\n)=-g(\ol\n)$. 
\item 
For $\a$, $\b\in\ck$, $f(\l)=\a(\l)\b(1-\l)$ if and only if 
$\wh f(\n)= - j(\a\ol\n,\b)$. 
\end{enumerate}
\end{epl}

\subsection{Hypergeometric functions}

\begin{dfn}
Define the {\em hypergeometric function} on $\k$ with parameters $\baa, \bbb \in P$ and a variable $\l$  by
$$F(\baa,\bbb;\l) = \frac{1}{1-q} \sum_{\n\in\ck} \frac{(\baa)_\n}{(\bbb)_\n^\0} \n(\l).$$
Note that $F(\baa,\bbb;0)=0$ by definition. 
It takes values in $\Q(\mu_{p(q-1)})$ in general. 
When $\deg(\baa)=\deg(\bbb)$, by Lemma \ref{l2.1} (iii), it takes values in $\Q(\mu_{q-1})$ and does not depend on the choice of $\psi$.  
When $\baa=\a_1+\cdots+\a_r$, $\bbb=\e+\b_1+\cdots+\b_s$, we also write 
$$F(\baa,\bbb;\l) =F\left({\a_1,\dots,\a_r\atop\b_1,\dots,\b_s};\l\right),$$
so that the analogy with the complex case be clear. 
We may also write this as ${}_rF_s$ to indicate the degrees ($r$ and $s+1$) of the parameters. 
\end{dfn}

Over the complex numbers, we have
$${}_1F_0\left({1\atop};x\right)=(1-x)^{-1}, \quad {}_0F_0(x)=e^x.$$
Their finite analogues are the following. 

\begin{ppn}\label{p2.1}\ 
\begin{enumerate}
\item
For any $\l\in\k$,  we have $F(0,0;\l)=-\d(1-\l)$. 
\item
For any $\l \in \k^*$, we have $F(0,\e;\l)=\psi(-\l)$. 
\end{enumerate}
\end{ppn}
\begin{proof}(i) Evident. 
(ii) By Example \ref{e2.2} (ii) and Proposition \ref{jacobi} (iii), the Fourier transform of $\psi(-\l)$ is 
$-\n(-1)g(\ol\n)=-1/(\e)^\0_\n$, 
and the result follows by the Fourier inversion. 
\end{proof}

We can shift simultaneously the parameters of $F(\baa,\bbb;\l)$. 
In particular, any hypergeometric function is reduced to a ${}_rF_s$-function. 
\begin{ppn}\label{shift}
For any $\baa, \bbb \in P$ and $\vp\in\ck$, we have 
$$F(\baa,\bbb;\l)= \frac{(\baa)_\vp}{(\bbb)^\0_\vp} \vp(\l) F(\baa\vp,\bbb\vp;\l).$$
Here we write 
$\baa\vp=\a_1\vp+\cdots + \a_r\vp$ when $\baa=\a_1+\cdots+\a_r$. 
\end{ppn}
\begin{proof}
This follows immediately from Lemma \ref{l2.1} (i). 
\end{proof}

Exchanging the numerator and denominator parameters results in the following.  
\begin{ppn}\label{reverse}
For any $\baa$, $\bbb \in P$ and $\l \in \k^*$, we have 
$$F(\bbb,\baa;\l)=F(\ol\baa,\ol\bbb; (-1)^{\deg(\baa+\bbb)}\l^{-1})
=\ol{F(\baa,\bbb;\l^{-1})}.$$
Here we write $\ol\baa=\ol\a_1+\cdots+\ol\a_r$ for $\baa=\a_1+\cdots+\a_r$. 
\end{ppn}
\begin{proof}The first (resp. second) equality 
follows  immediately from Lemma \ref{l2.1} (ii) (resp. Proposition \ref{jacobi} (ii)). 
\end{proof}

We have the following (cf. \cite[(8.2.8)]{katz}). 
\begin{ppn}
For any $\baa$, $\bbb \in P$, we have  
$$\sum_{\l\in\k} |F(\baa,\bbb;\l)|^2
=\frac{1}{q-1} \sum_{\n\in\ck} q^{(\baa+\bbb,\e)-(\baa+\bbb,\n)}.$$
\end{ppn}

\begin{proof}
Apply the Plancherel formula and use Proposition \ref{jacobi} (iii).  
\end{proof}

\begin{rmk}\label{comparison}
Let us compare our definition with other definitions in the literature. 
Let $\baa=\a_0+\cdots+\a_m$, $\bbb=\b_0+\cdots+\b_n \in P$. 
\begin{enumerate}
\item Koblitz \cite[Remark 2 after Theorem 3]{koblitz} considers the case where  
$m=n$, $\b_0=\e$ and $(\b_i,\e)=0$ for $i \ne 0$. 
His function ${}_{n+1}F_{n,\k}\left({\a_0,\a_1,\dots, \a_n\atop \b_1,\dots, \b_n}; \l\right)$
coincides with our $(-1)^n F(\baa,\bbb;\l)$ for $\l\in\k^*$ by Corollary \ref{c2.1} (i), but not for $\l=0$. 

\item Greene \cite[Definition 3.10]{greene} considers the case that $m=n$ and $\b_0=\e$. 
His definition uses ``binomial coefficients" 
$$\binom{\a}{\b}:=-\frac{\b(-1)}{q} j(\a,\ol\b)=-\frac{\a\b(-1)}{q} j(\a,\ol\a\b).$$ 
If  $\a\ne \b$, then 
$\binom{\a\n}{\b\n}/\binom{\a}{\b} = (\a)_\n/(\b)^\0_\n$ for any $\n$. 
Therefore, if $\a_i\ne \b_i$ for all $i$ (including $0$), then Greene's function  
${}_{n+1}F_n\left({\a_0,\a_1,\dots, \a_n\atop \phantom{\a_0, }\b_1,\dots, \b_n} \bigm|  \l\right)$ 
coincides with our 
$$\prod_{i=1}^n \binom{\a_i}{\b_i} F(\baa,\bbb;\l).$$  

\item Katz's hypergeometric sum \cite[(8.2.7)]{katz} 
$\mathrm{Hyp}(\psi; \baa;\bbb)(\k,\l)$ ($\l\ne 0$) coincides with our 
$$(-1)^{m+n+1} \frac{\prod_{i=0}^m g(\a_i)}{\prod_{j=0}^n (q^{-1}g^\0(\b_j))}  F(\baa,\bbb;\l^{-1})$$
(see Remark \ref{r3.1}).
\item
McCarthy \cite[Definition 1.4]{mccarthy} considers the case where $m=n$ and $\b_0=\e$. 
His function 
${}_{n+1}F_n\left({\a_0,\a_1,\dots, \a_n\atop \phantom{\a_0, }\b_1,\dots, \b_n}; \l\right)_q^*$ coincides with our $F(\baa,\bbb;\l)$ by Proposition \ref{jacobi}. 

\item In \cite{fuselieretal}, the authors consider the case where $m=n$ and $\b_0=\e$, and defines a function 
${}_{n+1}\F_n\left[{\a_0,\a_1,\dots, \a_n \atop \phantom{\a_0,} \b_1,\dots, \b_n}; \l\right]$. 
By Corollary \ref{c2.1} (i), it coincides with our $F(\baa,\bbb;\l)$ if $\a_0\ne\e$ and $\l\ne 0$. 
\end{enumerate}
We remark that the functions of \cite{koblitz}, \cite{greene}, \cite{fuselieretal} in general depend not only on $\baa$ and $\bbb$, but also on the orders of $\a_i$'s and of $\b_j$'s. 
\end{rmk}

\subsection{Other hypergeometric functions}

In the similar manner, one defines more general hypergeometric functions. 
For example, analogues of Lauricella's functions with $n$ variables are defined as follows. 

\begin{align*}
&F_A(\a,\b_1,\dots, \b_n,\g_1,\dots, \g_n;\l_1,\dots,\l_n) 
\\&\phantom{F_C(\a,\b,\g_1,\dots, \g_n;\l_1,\dots, \l_n)} 
= \frac{1}{(1-q)^n} \sum_{\n_1,\dots, \n_n \in \ck}\frac{(\a)_{\n_1\cdots\n_n} \prod_{i=1}^n (\b_i)_{\n_i}}{\prod_{i=1}^n(\e)^\0_{\n_i}(\g_i)^\0_{\n_i}}\prod_{i=1}^n \n_i(\l_i). 
\\&F_B(\a_1,\dots,\a_n,\b_1,\dots, \b_n,\g;\l_1,\dots, \l_n) 
\\&\phantom{F_C(\a,\b,\g_1,\dots, \g_n;\l_1,\dots, \l_n)} 
= \frac{1}{(1-q)^n}  \sum_{\n_1,\dots, \n_n \in \ck}\frac{\prod_{i=1}^n (\a_i)_{\n_i}(\b_i)_{\n_i}}{\prod_{i=1}^n(\e)^\0_{\n_i}(\g)^\0_{\n_1\cdots \n_n}}\prod_{i=1}^n \n_i(\l_i). 
\\&F_C(\a,\b,\g_1,\dots, \g_n;\l_1,\dots, \l_n) 
= \frac{1}{(1-q)^n} \sum_{\n_1,\dots, \n_n \in \ck}\frac{(\a)_{\n_1\cdots \n_n}(\b)_{\n_1\cdots\n_n}}{\prod_{i=1}^n(\e)^\0_{\n_i}(\g_i)^\0_{\n_i}}\prod_{i=1}^n \n_i(\l_i). 
\\&F_D(\a,\b_1,\dots, \b_n,\g;\l_1,\dots,\l_n) 
=\frac{1}{(1-q)^n}  \sum_{\n_1,\dots, \n_n \in \ck}\frac{(\a)_{\n_1\cdots\n_n} \prod_{i=1}^n (\b_i)_{\n_i}}{\prod_{i=1}^n(\e)^\0_{\n_i}(\g)^\0_{\n_1\cdots\n_n}}\prod_{i=1}^n \n_i(\l_i). 
\end{align*}
These are all functions on $\k^n$ with values in $\Q(\mu_{q-1})$ independent of the choice of $\psi$ by
$(\a)_{\n_1\cdots \n_n}=\prod_{i=1}^n (\a\n_1\cdots \n_{i-1})_{\n_i}$ 
and Lemma \ref{l2.1} (iii).  

When $n=2$, these are analogues of Appell's functions $F_2$, $F_3$, $F_4$, $F_1$, respectively, 
and the definition as above coincides with that of Barman--Saikia--Tripathi \cite{bst}.

\section{Basic properties}
\subsection{Reduction and iteration}

If a complex hypergeometric function has common parameters in the numerator and the denominator, they cancel out by definition. 
This is not the case for our finite version. 

\begin{dfn}
We say that $F(\baa,\bbb;\l)$ is {\em reduced} if $(\baa,\bbb)=0$.  
For general $\baa$ and $\bbb$, let $\bgg\in P$ be the element of largest degree such that 
$\baa-\bgg$, $\bbb-\bgg \in P$. Then $(\baa-\bgg,\bbb-\bgg)=0$, and 
we define the {\em reduction} of $F(\baa,\bbb;\l)$ by 
$$\wt F(\baa,\bbb;\l):=F(\baa-\bgg,\bbb-\bgg;\l).$$
\end{dfn}

The relation between $F(\baa,\bbb;\l)$ and $\wt F(\baa,\bbb;\l)$ is given by the following Theorem.
 
\begin{thm}\label{reduction}
For any $\baa,\bbb,\bgg \in P$, we have 
\begin{equation*}
F(\baa+\bgg,\bbb+\bgg;\l)=q^{(\bgg,\e)}\left(F(\baa,\bbb;\l)+q^{-1}\sum_{\n\in\ck} 
\frac{1-q^{-(\n,\bgg)}}{1-q^{-1}} \frac{(\baa)_{\ol\n}}{(\bbb)^\0_{\ol\n}} \ol\n(\l)
\right). 
\end{equation*}
\end{thm}

\begin{proof}
If $\g\in\ck$, then we have \begin{align*}
F(\baa+\g,\bbb+\g;\l)
&=q^{\d(\g)}\frac{1}{1-q} \sum_\n q^{-\d(\g\n)}\frac{(\baa)_\n}{(\bbb)^\0_\n}\n(\l)
\\&=q^{\d(\g)}\frac{1}{1-q} 
\left(\sum_\n \frac{(\baa)_\n}{(\bbb)^\0_\n}\n(\l) + (q^{-1}-1) \frac{(\baa)_{\ol\g}}{(\bbb)^\0_{\ol\g}}\ol\g(\l)\right)
\\&=q^{\d(\g)}\left(F(\baa,\bbb;\l)+q^{-1}\frac{(\baa)_{\ol\g}}{(\bbb)^\0_{\ol\g}}\ol\g(\l)\right).
\end{align*}
The general case follows by induction on $\deg(\bgg)$. 
\end{proof}

We have the iteration formula for the complex hypergeometric function  
$$B(a,b-a){}_{r+1}F_{s+1}\left({a_1,\dots, a_r,a \atop b_1,\dots, b_s, b};x\right)
=\int_0^1 {}_rF_s\left({a_1,\dots, a_r \atop b_1,\dots, b_s};x t\right)t^{a-1}(1-t)^{b-a-1}\, dt$$
under a suitable convergence condition (cf. \cite[(4.1.1)]{slater}). 
Its finite analogue is the following. 

\begin{thm}\label{iteration}
Suppose that $\baa$, $\bbb \in P$, $\a$, $\b \in \ck$ and $\a \ne \b$. Then for any $\l\in\k$, 
\begin{equation*}
-j(\a,\ol\a\b)F(\baa+\a,\bbb+\b;\l)=\sum_{t \in \k} F(\baa,\bbb;\l t) \a(t)\ol\a\b(1-t).
\end{equation*}
\end{thm}

\begin{proof}The right-hand side equals
$$\frac{1}{1-q}\sum_\n \frac{(\baa)_\n}{(\bbb)^\0_\n} \n(\l)\sum_t \a\n(t)\ol\a\b(1-t)
=-\frac{1}{1-q}\sum_\n \frac{(\baa)_\n}{(\bbb)^\0_\n} \n(\l) j(\a\n,\ol\a\b).$$
By Proposition \ref{jacobi} (iv),  
$$j(\a\n,\ol\a\b)=\frac{g(\a\n)g(\ol\a\b)}{g^\0(\b\n)}=\frac{g(\a)g(\ol\a\b)}{g^\0(\b)}\frac{(\a)_\n}{(\b)^\0_\n}=j(\a,\ol\a\b)\frac{(\a)_\n}{(\b)^\0_\n},$$
hence the formula follows. 
\end{proof}

Recall that over the complex numbers, ${}_1F_0(x)$ is a geometric series 
$${}_1F_0\left({a\atop };x\right)=(1-x)^{-a}.$$
Analogously, we have the following.

\begin{cor}\label{c2.0}
Suppose that $\a \in \ck$ and $\a \ne \e$. Then for any $\l \in \k^*$, 
$$\FF{\a}{}{\l}=\ol\a(1-\l).$$
\end{cor}
\begin{proof}By Theorem \ref{iteration} and Proposition \ref{p2.1} (i), 
\begin{align*}
&-j(\a,\ol\a) \FF{\a}{}{\l}= -\sum_{t\in\k} \d(1-\l t)\a(t)\ol\a(1-t)
\\&=-\a(\l^{-1})\ol\a(1-\l^{-1})=-\ol\a(\l-1). 
\end{align*}
Since $j(\a,\ol\a)=\a(-1)$ by Proposition \ref{jacobi}, the formula follows. 
\end{proof}

As a result, we obtain the following sum representations of hypergeometric functions. 

\begin{cor}\label{c2.1}
Suppose that $\baa=\a_1+\cdots+\a_d$, $\bbb=\b_1+\cdots +\b_d \in P_d$ ($d\ge 0$) and $\a_i\ne \b_i$ for all $i$. 
\begin{enumerate}
\item 
For any $\a_0 \in \ck$, $\a_0 \ne \e$, and $\l \in \k^*$, 
\begin{align*}
& \prod_{i=1}^d (-j(\a_i,\ol\a_i\b_i)) \cdot  \FF{\a_0,\a_1,\dots,\a_d}{\b_1,\dots, \b_d}{\l}
\\& =\sum_{t_1,\dots, t_d\in \k}\ol\a_0(1-\l t_1\cdots t_d) \prod_{i=1}^d \a_i(t_i)\ol\a_i\b_i(1-t_i).
\end{align*}
\item 
For any $\l\in\k$, 
$$\prod_{i=1}^d (-j(\a_i,\ol\a_i\b_i)) \cdot F(\baa,\bbb;\l)
= - \sum_{\l t_1\cdots t_{d}=1}  
 \prod_{i=1}^{d} \a_i(t_i)\ol\a_i\b_i(1-t_i).$$
\end{enumerate}
\end{cor}

\begin{proof}For (i) (resp. (ii)), 
start with Corollary \ref{c2.0} (resp. Proposition \ref{p2.1} (i)) and apply Theorem \ref{iteration} iteratively.   
\end{proof}

We also have the iteration formulas under suitable convergence conditions
\begin{align*}
\G(a) {}_{r+1}F_s\left({a_1,\dots, a_r,a \atop b_1,\dots, b_s};x\right)
&=
\int_0^\infty {}_rF_s\left({a_1,\dots, a_r \atop b_1,\dots, b_s};xt\right)e^{-t}t^{a-1}\, dt, 
\\
\frac{2\pi i}{\G(b)} {}_{r}F_{s+1}\left({a_1,\dots, a_r\atop b_1,\dots, b_s,b};x\right)
&=
\int_\g {}_rF_s\left({a_1,\dots, a_r \atop b_1,\dots, b_s};xt^{-1} \right)e^{t}t^{-b}\, dt,
\end{align*}
where $\g$ is the Hankel contour multiplied with $-1$. 
Their finite analogues are the following. 

\begin{thm}\label{iteration2}
Suppose that $\baa, \bbb \in P$, $\a,\b \in \ck$. Then for any $\l\in\k$, 
\begin{align*}
-g(\a) F(\baa+\a,\bbb;\l)& = \sum_{t\in\k} F(\baa,\bbb;\l t) \psi(t)\a(t),
\\
-\frac{q}{g^\0(\b)} F(\baa,\bbb+\b;\l)& = \sum_{t\in\k} F(\baa,\bbb;\l t^{-1}) \psi(-t)\ol\b(t). 
\end{align*}
\end{thm}

\begin{proof}For each $\n\in\ck$, 
$$\sum_{t\in\k}\psi(t)\a(t)\n(\l t)=-g(\a\n)\n(\l)=-g(\a)(\a)_\n\n(\l),$$
and the first formula follows. The second formula follows similarly, using Proposition \ref{jacobi} (iii). 
\end{proof}

\begin{cor}\label{iteration3}
For any $\baa=\a_1+\cdots+\a_m$, $\bbb=\b_1+\cdots +\b_n \in P$ and $\l\in\k$, \
$$\frac{\prod_{i=1}^m (-g(\a_i))}{\prod_{j=1}^n (-q^{-1}g^\0(\b_j))} \cdot 
F(\baa,\bbb;\l)=-\sum_{\l s_1\cdots s_m=t_1\cdots t_n} \prod_{i=1}^m \psi(s_i)\a_i(s_i) \prod_{j=1}^n \psi(-t_j)\ol\b_j(t_j).
$$
\end{cor}
\begin{proof}
Start with Proposition \ref{p2.1} (i) and apply Theorem \ref{iteration2} iteratively.   
\end{proof}

\begin{rmk}\label{r3.1}
By Corollary \ref{iteration3}, a function of the form $F(\baa,0;\l)$ or $F(0,\bbb;\l)$ is essentially given by generalized Kloosterman sums, defined as 
$$\mathrm{Kl}(\psi;\a_1,\dots,\a_d;1,\dots, 1)(\k,\l)=
\sum_{s_1,\dots, s_d \in\k, s_1 \cdots s_d=\l} \prod_{i=1}^d \psi(s_i)\a_i(s_i)$$
using the notation of Katz \cite[4.0]{katz2}. 
The right-hand side of Corollary \ref{iteration3} is Katz's definition of $-\operatorname{Hyp}(\psi;\baa,\bbb)(\k,\l^{-1})$ (see Remark \ref{comparison} (iii)). 
\end{rmk}

\begin{epl}\label{e3.1}
Here are some examples of non-reduced functions. Let $\baa, \bbb \in P$, $\a, \b \in \ck$ and $\l\in\k^*$. 
\begin{enumerate}
\item 
$F(\baa+\e,\bbb+\e;\l)=q F(\baa,\bbb;\l)+1.$
\item
$F(\a,\a;\l)= q^{\d(\a)}\left(-\d(1-\l)+q^{-1} \ol\a(\l)\right).$
\item
$
\displaystyle
\FF{\a,\b}{\b}{\l}=
\begin{cases}
\displaystyle
q^{\d(\b)}\ol\a(1-\l)+\frac{g(\a\ol\b)}{g(\a)g(\ol\b)}\ol\b(\l) & (\a\ne \e), 
\\
\displaystyle
q^{\d(\b)}\left(-q\d(1-\l)+1\right) + \ol\b(\l) & (\a=\e).
\end{cases}$
\end{enumerate}
\end{epl}

\subsection{Multiplication formula}

Recall the multiplication formula for the gamma function
$$\G(nx)=(2\pi)^\frac{1-n}{2} n^{nx-\frac{1}{2}}\prod_{i=0}^{n-1} \G\left(x+\frac{i}{n}\right),$$
from which follows 
$$\frac{\G(nx)}{\G(n)}=n^{n(x-1)} \prod_{i=0}^{n-1} \frac{\G\left(x+\frac{i}{n}\right)}{\G\left(1+\frac{i}{n}\right)}. 
$$
Its finite analogue is the following. 
In the sequel, we assume that $n \mid q-1$. 

\begin{thm}[Davenport--Hasse \cite{d-h}]\label{d-h}
For any $\a\in \wh{\k^\times}$, 
$$g(\a^n)=\a^n(n)\prod_{\vp\in\wh{\k^\times}, \vp^n=\e} \frac{g(\a\vp)}{g(\vp)}. $$
\end{thm}

We give an elementary proof in the appendix. 
See \cite[11.3]{berndtetal}, \cite[Theorem 3]{terasoma} for other proofs. 
As a corollary, we obtain multiplication formulas for Pochhammer symbols. 

\begin{cor}\label{multiplication}
For any $\a$, $\n \in \ck$, 
$$(\a^n)_{\n^n}=\n^n(n) \prod_{\vp^n=\e} (\a\vp)_\n, \quad 
(\a^n)^\0_{\n^n}=\n^n(n)\prod_{\vp^n=\e} (\a\vp)^\0_\n.$$
\end{cor}

We will use frequently the duplication formulas 
$$g(\a^2)=\a(4) \frac{g(\a)g(\a\phi)}{g(\phi)},
\ (\a^2)_{\n^2}=\n(4)(\a)_\n(\a\phi)_\n, 
\ (\a^2)^\0_{\n^2}=\n(4)(\a)^\0_\n(\a\phi)^\0_\n,$$
where $\phi$ is the quadratic character (i.e. $\phi^2=\e, \phi\ne \e$). 

Theorem \ref{d-h} is rephrased in terms of hypergeometric functions. 
The following corollary is in fact equivalent to the theorem by the Fourier transform. 

\begin{cor}\label{dwork}
For any $\l\in\k^*$, 
\begin{align*}
(1-q) F\left(\textstyle\sum_{\vp^n=\e} \vp, n\e;n^n\l\right)
&=
1+q \sum_{\n\ne\e} j(\underbrace{\ol\n,\ol\n,\dots, \ol\n}_{\text{$n$ times}}) \n(\l),
\\
(1-q) \wt F\left(\textstyle\sum_{\vp^n=\e} \vp, n\e;n^n\l\right)
& =1+ \sum_{\n\ne\e} j(\underbrace{\ol\n,\ol\n,\dots, \ol\n}_{\text{$n$ times}}) \n(\l). 
\end{align*}
\end{cor}

\begin{proof}By Corollary \ref{multiplication} and  Proposition \ref{jacobi}, 
$$\left(\prod_{\vp^n=\e} \frac{(\vp)_\n}{(\e)^\0_\n}\right)\n(n^n)
=q^n \frac{g(\n^n)}{g^\0(\n)^n}=q \frac{g(\ol\n)^n}{g^\0(\ol{\n}^n)}
=\begin{cases} 
1 & (\n=\e),\\
qj(\underbrace{\ol\n,\ol\n,\dots, \ol\n}_{\text{$n$ times}}) & (\n\ne \e), 
\end{cases}$$ 
and the first formula follows, from which the second one follows by Example \ref{e3.1} (i). 
\end{proof}

\begin{rmk}\label{r.dwork}
The Dwork hypersurface of degree $n$ is defined by the homogeneous equation 
$$x_1^n+\cdots + x_n^n = n \l  x_1\cdots x_n. $$
The values in Corollary \ref{dwork} describe the trace of Frobenius acting on a 
$(n-1)$-dimensional subspace of the middle $l$-adic cohomology (Nakagawa \cite{nakagawa}).  
\end{rmk}

\subsection{Linear transformations}

Recall the transformation formulas for complex Gauss functions due respectively to 
 Euler and Pfaff
\begin{align*}
{}_2F_1\left({a,b\atop c};x\right) &= (1-x)^{c-a-b} {}_2F_1\left({c-a,c-b\atop c};x\right), 
\\{}_2F_1\left({a,b\atop c};x\right) &= (1-x)^{-a} {}_2F_1\left({a,c-b\atop c};\frac{x}{x-1}\right). 
\end{align*}
We have the following finite analogues (cf. \cite[Theorem 4.4 (iv), (ii)]{greene}). 

\begin{thm}\label{p2.5} Suppose that $(\a+\b,\e+\g)=0$. 
\begin{enumerate}
\item
For any $\l\ne 1$, 
$$\FF{\a,\b}{\g}{\l} =\ol{\a\b}\g(1-\l) \FF{\ol\a\g, \ol\b\g}{\g}{\l}.$$
\item
For any $\l\ne 1$, 
$$\FF{\a,\b}{\g}{\l} =\ol\a(1-\l) \FF{\a, \ol\b\g}{\g}{\frac{\l}{\l-1}}.$$
\end{enumerate}
\end{thm}

\begin{proof}
(i) By Corollary \ref{c2.1} (i), 
\begin{align*}
F(\a+\b,\e+\g;\l)&= -j(\b,\ol\b\g)^{-1} \sum_t \ol\a(1-\l t)\b(t)\ol{\b}\g(1-t), 
\\\ol{\a\b}\g(1-\l) F(\ol\a\g+\ol\b\g,\e+\g;\l)&= -j(\ol\b\g,\b)^{-1} \sum_s \a\ol\g(1-\l s) \ol\b\g(s)\b(1-s). 
\end{align*}
Letting $t=\frac{1-s}{1-\l s}$, the right members agree. 
The statement (ii) is proved similarly by letting $t=1-s$. 
\end{proof}

Recall that the complex function ${}_2F_1\left({a,b\atop c};x\right)$ is a solution of the second-order differential equation
$$ \left[ \frac{d^2}{dx^2}+\left(\frac{c}{x}-\frac{a+b-c+1}{1-x}\right)\frac{d}{dx} -\frac{ab}{x(1-x)}\right] y=0.$$
Obviously, ${}_2F_1\left({a,b\atop a+b-c+1};1-x\right)$ is another solution. 
These two functions are generically linearly independent, and by iterating the Euler and Pfaff transformations, we obtain Kummer's 24 solutions  around the singularities $0$, $1$, $\infty$ (cf. \cite[1.3]{slater}). 

Over a finite field, the corresponding two functions are no longer linearly independent and we have the following (cf. \cite[Theorem 4.4 (i)]{greene}). 

\begin{thm}\label{connection}
Suppose that $(\a+\b,\e+\g)=0$. Then for any $\l \ne 0$, $1$, 
$$\FF{\a,\b}{\g}{\l}=\frac{g^\0(\g)g(\ol{\a\b}\g)}{g(\ol\a\g)g(\ol\b\g)}\FF{\a,\b}{\a\b\ol\g}{1-\l}.$$
\end{thm}
\begin{proof}
This is proved similarly as Theorem \ref{p2.5} (i)  by letting $t=\frac{s}{s-1}$, together with 
$$\frac{\b(-1)j(\b,\a\ol\g)}{j(\b,\ol\b\g)}=\frac{\b(-1)g(\a\ol\g)g^\0(\g)}{g(\ol\b\g)g^\0(\a\b\ol\g)}=\frac{g^\0(\g)g(\ol{\a\b}\g)}{g^\0(\ol\a\g)g(\ol\b\g)}$$
using Proposition \ref{jacobi}. 
\end{proof}

Combining Theorem \ref{p2.5} and Theorem \ref{connection}, we obtain the following relations among the analogues of Kummer's 24 solutions.  

\begin{cor}\label{c2.3}
Suppose that $(\a+\b,\e+\g)=0$. Then for any $\l \ne 0, 1$, 
\begin{align*}
& \FF{\a,\b}{\g}{\l}=\ol{\a\b}\g(1-\l) \FF{\ol\a\g, \ol\b\g}{\g}{\l}
\\& =G_1\ol\g(\l)\ol{\a\b}\g(1-\l) \FF{\ol\a,\ol\b}{\ol\g}{\l}  =G_1 \ol\g(\l)\FF{\a\ol\g,\b\ol\g}{\ol\g}{\l} 
\\& =G_2 \FF{\a,\b}{\a\b\ol\g}{1-\l}
=G_2 \ol\g(\l) \FF{\a\ol\g,\b\ol\g}{\a\b\ol\g}{1-\l}
\\& =G_3\ol\g(\l)\ol{\a\b}\g(1-\l)  \FF{\ol\a,\ol\b}{\ol{\a\b}\g}{1-\l}
=G_3 \ol{\a\b}\g(1-\l) \FF{\ol\a\g,\ol\b\g}{\ol{\a\b}\g}{1-\l}
\\& = G_4 \ol\a(-\l) \FF{\a,\a\ol\g}{\a\ol\b}{\frac{1}{\l}} 
=  G_4  \b\ol\g(-\l) \ol{\a\b}\g(1-\l) \FF{\ol\b,\ol\b\g}{\a\ol\b}{\frac{1}{\l}}
\\& 
= G_5 \ol\b(-\l) \FF{\b,\b\ol\g}{\ol\a\b}{\frac{1}{\l}}
=  G_5  \a\ol\g(-\l) \ol{\a\b}\g(1-\l) \FF{\ol\a,\ol\a\g}{\ol\a\b}{\frac{1}{\l}}
\\& =\ol\a(1-\l) \FF{\a, \ol\b\g}{\g}{\frac{\l}{\l-1}}=\ol\b(1-\l) \FF{\ol\a\g, \b}{\g}{\frac{\l}{\l-1}}
\\&=G_1 \ol\g(\l)\ol\b\g(1-\l)\FF{\ol\a,\b\ol\g}{\ol\g}{\frac{\l}{\l-1}}
= G_1 \ol\g(\l)\ol\a\g(1-\l)\FF{\a\ol\g,\ol\b}{\ol\g}{\frac{\l}{\l-1}}
\\&= G_2 \ol\a(\l) \FF{\a,\a\ol\g}{\a\b\ol\g}{\frac{\l-1}{\l}}
=G_2 \ol\b(\l) \FF{\b,\b\ol\g}{\a\b\ol\g}{\frac{\l-1}{\l}}
\\&= G_3 \a\ol\g(\l)\ol{\a\b}\g(1-\l) \FF{\ol\a,\ol\a\g}{\ol{\a\b}\g}{\frac{\l-1}{\l}}
= G_3 \b\ol\g(\l)\ol{\a\b}\g(1-\l) \FF{\ol\b,\ol\b\g}{\ol{\a\b}\g}{\frac{\l-1}{\l}}
\\& =G_4\ol\a(1-\l) \FF{\a,\ol\b\g}{\a\ol\b}{\frac{1}{1-\l}}
=G_4 \ol\g(-\l)\ol\a\g(1-\l) \FF{\a\ol\g,\ol\b}{\a\ol\b}{\frac{1}{1-\l}}
\\& =G_5 \ol\b(1-\l) \FF{\ol\a\g,\b}{\ol\a\b}{\frac{1}{1-\l}}
=G_5 \ol\g(-\l) \ol\b\g(1-\l) \FF{\ol\a,\b\ol\g}{\ol\a\b}{\frac{1}{1-\l}}, 
\end{align*}
where
\begin{gather*}
G_1=\frac{g(\a\ol\g)g(\b\ol\g)g(\g)}{g(\a)g(\b)g(\ol\g)}, \quad 
G_2=\frac{g^\0(\g)g(\ol{\a\b}\g)}{g(\ol\a\g)g(\ol\b\g)}, \quad 
G_3=\frac{g^\0(\g)g({\a\b}\ol\g)}{g(\a)g(\b)}, 
\\G_4=\frac{g^\0(\g)g(\ol\a\b)}{g(\ol\a\g)g(\b)}, \quad 
G_5=\frac{g^\0(\g)g(\a\ol\b)}{g(\ol\b\g)g(\a)}.
\end{gather*}
(These satisfy  $q^{(\g,\e)}G_1=q^{(\a\b,\g)} G_2G_3=q^{(\a,\b)} \g(-1) G_4G_5$. )
\qed
\end{cor}

\section{Summation formulas}

Classically, the special values of 
${}_{d+1}F_d\left({a_0,a_1,\dots, a_d \atop b_1,\dots b_d};x\right)$, in particular the value at $x=\pm 1$, 
have been of  particular interest (cf. \cite{bailey}). 
Recall that 
${}_{d+1}F_d(x)$ converges at $x=1$ (resp. $x=-1$) if 
$\Re\bigl(\sum_{i=1}^d b_i-\sum_{i=0}^d a_i\bigr)>0$ (resp. $\Re\bigl(\sum_{i=1}^d b_i-\sum_{i=0}^d a_i\bigr)>-1$) (cf. \cite[1.1.1, 2.2]{slater}). 
These conditions are always assumed when we mention such special values. 

The function ${}_{d+1}F_d(x)$ is said to be Saalsch\"utzian if 
$\sum_{i=1}^d b_i=1+\sum_{i=0}^d a_i$. 
It is said to be well-poised (resp. nearly-poised) if all (resp. all but one) $a_i+b_i$ ($i=0,1,\dots, d$, setting $b_0=1$) agree 
 for a suitable ordering of $a_i$'s and $b_i$'s.  

\begin{dfn}Let $\baa=\a_1+\cdots+\a_d$, $\bbb=\b_1+\cdots+\b_d \in P$. 
We say that the function $F(\baa,\bbb;\l)$ is {\em Saalsch\"utzian} if
$\a_1\cdots \a_d=\b_1\cdots \b_d$. 
It is said to be {\em well-poised} (resp. {\em nearly-poised}) if all (resp. all but one) $\a_i\b_i$ agree for a suitable ordering of $\a_i$'s and $\b_i$'s.  
\end{dfn}

\subsection{Special values of ${}_1F_0$}

\begin{ppn}For $\a\in\ck$, 
\begin{align*}
F\left({\a\atop\ };1\right)
&=\begin{cases} 0 & (\a\ne \e), \\ 1-q& (\a=\e),  
\end{cases}
\\
F\left({\a\atop\ };-1\right)
&=\begin{cases} \ol\a(2) & (\a\ne \e), \\ 1-\d(2) q& (\a=\e). 
\end{cases}
\end{align*}
\end{ppn}
\begin{proof}See Corollary \ref{c2.0} and Example \ref{e3.1} (ii) when $\a=\e$. 
\end{proof}

\subsection{Special values of ${}_2F_1$}

Recall the Euler--Gauss summation formula  (cf. \cite[1.3]{bailey})
$${}_2F_1\left({a,b \atop c};1\right)=\frac{\G(c)\G(c-a-b)}{\G(c-a)\G(c-b)}. $$
Its finite analogue is the following (cf. \cite[Th\'eor\`eme 1 (i)]{helversen-pasotto}, \cite[Theorem 4.9]{greene}, \cite[Theorem 1.9]{mccarthy}). 

\begin{thm}\label{summation}
For any $\a$, $\b$, $\g \in \ck$, 
$$\FF{\a,\b}{\g}{1}=
\begin{cases}
\dfrac{g^\0(\g)g(\ol{\a\b}\g)}{g^\0(\ol\a\g)g^\0(\ol\b\g)} & (\a+\b\ne \e+\g), \\
1+q^{\d(\g)}(1-q) & (\a+\b = \e+\g).
\end{cases}$$
\end{thm}

\begin{proof}
First, if $\a\ne\e$ and $\b\ne\g$, then by Corollary \ref{c2.1} (i) and Proposition \ref{jacobi},  
$$\FF{\a,\b}{\g}{1}=-j(\b,\ol\b\g)^{-1} \sum_{t\in\k}\b(t)\ol{\a\b}\g(1-t)
=\frac{j(\b,\ol{\a\b}{\g})}{j(\b,\ol\b\g)}
=\frac{g^\0(\g)g(\ol{\a\b}\g)}{g^\0(\ol\a\g)g(\ol\b\g)}. $$
Secondly if $\a=\e$ and $\b\ne\g$, then by Example \ref{e3.1} (i), Proposition \ref{shift} and Corollary  \ref{c2.0}, 
$$\FF{\a,\b}{\g}{1}=qF(\b,\g;1)+1=1,$$
hence the formula. Finally if $\b=\g$, then the formula follows from Example \ref{e3.1} (iii) using Proposition \ref{jacobi}. 
\end{proof}

\begin{rmk}\label{r.summation}\ 
\begin{enumerate}
\item The second case $\a+\b=\e+\g$ of the theorem can be expressed as 
$$\FF{\a,\b}{\g}{1}= \dfrac{g^\0(\g)g(\ol{\a\b}\g)}{g^\0(\ol\a\g)g^\0(\ol\b\g)}-(1+\d(\g)q)\frac{(1-q)^2}{q}.$$
\item
Recall Vandermonde's theorem 
$${}_2F_1\left({a,-n \atop c};1\right)=\frac{(c-a)_n}{(c)_n} \quad (n\in\N).$$
If $(\a,\e+\g)=0$, then the theorem is written as 
$$\FF{\a,\ol\n}{\g}{1}=\frac{(\ol\a\g)_\n}{(\g)^\0_\n} \quad (\n\in\ck). $$
\end{enumerate}
\end{rmk}

Recall the multinomial theorem for Pochhammer symbols
$$\sum_{n_1+\cdots+n_d=n} \prod_{i=1}^d \frac{(a_i)_{n_i}}{(1)_{n_i}}
= \frac{(a_1+\cdots+a_d)_n}{(1)_n}.
$$
Its finite analogue is the following. 

\begin{cor}
Suppose that $\a_1,\dots, \a_d \in\ck$ and $\prod_{i=1}^j \a_i\ne\e$ for all $j=2,\dots, d$. 
Then, 
$$\sum_{\n_1\cdots \n_d=\n} \prod_{i=1}^d \left(\frac{1}{1-q} \frac{(\a_i)_{\n_i}}{(\e)^\0_{\n_i}}\right) =  \frac{1}{1-q} \frac{(\a_1\cdots\a_d)_\n}{(\e)^\0_\n}.
$$
\end{cor}

\begin{proof}
By induction, it suffices to prove the case $d=2$. Then, by Lemma \ref{l2.1},  
\begin{align*}
\sum_\m \frac{(\a)_\m}{(\e)^\0_\m} \frac{(\b)_{\ol\m\n}}{(\e)^\0_{\ol\m\n}}
&= \sum_\m \frac{(\a)_\m}{(\e)^\0_\m} 
\frac{(\e)_{\ol\n}(\ol\n)_\m}{(\ol\b)^\0_{\ol\n}(\ol{\b\n})^\0_\m}
= (1-q) \frac{(\b)_{\n}} {(\e)^\0_{\n}}\FF{\a,\ol\n}{\ol{\b\n}}{1}.
\end{align*}
If $\a\b\ne\e$, then $\a+\ol\n\ne\e+\ol{\b\n}$ for any $\n$, and 
$$ \FF{\a,\ol\n}{\ol{\b\n}}{1}=\frac{(\ol\b)^\0_{\ol\n}}{(\ol{\a\b})^\0_{\ol\n}}
=\frac{(\a\b)_\n}{(\b)_\n}$$
by Theorem \ref{summation}, hence the formula.
\end{proof}

Recall Kummer's formula (cf. \cite[2.3]{bailey}) for  well-poised ${}_2F_1(-1)$ 
$${}_2F_1\left({2a,b \atop 2a-b+1}; -1\right)=\frac{\G(2a-b+1)\G(a+1)}{\G(2a+1)\G(a-b+1)}.$$
Its finite analogue is the following (cf. \cite[(4.11)]{greene} and \cite[Theorem 1.10]{mccarthy}). 

\begin{thm}\label{t2.3}
Let $\a, \b \in \wh{\k^\times}$. 
\begin{enumerate}
\item If $p=2$, then 
$$\FF{\a^2,\b}{\a^2\ol\b}{-1}
= \begin{cases} \dfrac{g^\0(\a^2\ol\b)g(\a)}{g(\a^2)g^\0(\a\ol\b)} & (\b\ne\e), \\
1+q^{\d(\a)}(1-q) & (\b=\e).
\end{cases}$$
\item If $p$ is odd,  then
$$\FF{\a^2,\b}{\a^2\ol\b}{-1}
= \sum_{\a'^2=\a^2} \frac{g^\0(\a^2\ol\b)g(\a')}{g(\a^2)g^\0(\a'\ol\b)}.$$
\item 
If $\a$ is not a square, then 
$$\FF{\a,\b}{\a\ol\b}{-1}=0.$$
\end{enumerate}

\end{thm}
\begin{proof}
(i) This is equivalent to Theorem \ref{summation}, by Proposition \ref{jacobi} (iv). 
Note that 
$j(\a,\ol\b)=j(\a^2,\ol\b^2)$ by the Frobenius automorphism of $\k$, and that $\d(\a)=\d(\a^2)$. 

(ii) When $\b \ne \e$, we have by Corollary \ref{c2.1} (i)
\begin{align*}
&-j(\a^2,\ol\b)\FF{\a^2,\b}{\a^2\ol\b}{-1}
=\sum_{t\in\k} \ol\b(1+t)\a^2(t)\ol\b(1-t)
=\sum_{t\in\k} \a(t^2)\ol\b(1-t^2)
\\&
=\sum_{s\in\k} \left(1+\phi(s)\right)\a(s)\ol\b(1-s)=-j(\a,\ol\b)-j(\a\phi,\ol\b),
\end{align*}
and the formula follows by Proposition \ref{jacobi} (iv). 
When $\b =\e$, we have by Example \ref{e3.1} (iii) 
$$\FF{\a^2,\e}{\a^2}{-1}= q^{\d(\a^2)}+1 =q^{\d(\a)}+q^{\d(\a\phi)},$$
hence the formula. 

(iii) By Propositions \ref{reverse}, \ref{shift} and \ref{jacobi}, 
$$F(\a+\b,\e+\a\ol\b;-1)=F(\e+\ol\a\b,\ol\a+\ol\b;-1)=\a(-1)F(\a+\b,\e+\a\ol\b;-1).$$
Since $\a(-1)=-1$ by assumption, the assertion follows. 
\end{proof}

\subsection{Special values of ${}_3F_2$} 
Over the complex numbers, a fundamental theorem on ${}_3F_2(1)$ is Thomae's formula (cf. \cite[3.2 (1)]{bailey}) 
$$
\frac{\G(a)}{\G(d)\G(e)}{}_3F_2\left({a,b,c \atop d,e};1\right)
=\frac{\G(s)}{\G(b+s)\G(c+s)} {}_3F_2\left({s,d-a,e-a \atop b+s,c+s};1\right),$$
where $s:=d+e-a-b-c$. 
The following is a finite analogue.

\begin{thm}\label{thomae}
Suppose that $(\a, \vp+\psi)=(\e,\b+\g)=0$. 
Then
$$\frac{g(\a)}{g^\0(\vp)g^\0(\psi)} \FF{\a,\b,\g}{\vp,\psi}{1}=
\frac{g(\s)}{g^\0(\b\s)g^\0(\g\s)} \FF{\s,\ol\a\vp,\ol\a\psi}{\b\s, \g\s}{1},$$
where $\s:=\ol{\a\b\g}\vp\psi$. 
\end{thm}

\begin{proof}The left-hand side times $1-q$ is 
\begin{align*}
& \sum_\n \frac{g(\a\n)}{g^\0(\vp\n)g^\0(\psi\n)} \frac{(\b+\g)_\n}{(\e)^\0_\n}
\\&=
\frac{1}{g^\0(\ol\a\vp\psi)}
\sum_\n \frac{g^\0(\ol\a\vp\psi\n)g(\a\n)}{g^\0(\vp\n)g^\0(\psi\n)} \frac{(\b+\g)_\n}{(\e+\ol\a\vp\psi)^\0_\n}
\\&=\frac{1}{g^\0(\ol\a\vp\psi)}\sum_\n \FF{\ol\a\vp,\ol\a\psi}{\ol\a\vp\psi\n}{1} \frac{(\b+\g)_\n}{(\e+\ol\a\vp\psi)^\0_\n}
\\&=\frac{1}{(1-q)g^\0(\ol\a\vp\psi)}\sum_{\m,\n} \frac{(\ol\a\vp+\ol\a\psi)_\m}{(\e+\ol\a\vp\psi\n)^\0_\mu}
\frac{(\b+\g)_\n}{(\e+\ol\a\vp\psi)^\0_\n}
\\&=\frac{1}{(1-q)g^\0(\ol\a\vp\psi)}\sum_{\m,\n} \frac{(\ol\a\vp+\ol\a\psi)_\m(\b+\g)_\n}{(\e)^\0_\m(\e)^\0_\n(\ol\a\vp\psi)^\0_{\m\n}}. 
\end{align*}
For the second equality, we used Theorem \ref{summation} and the assumption $(\a, \vp+\psi)=0$. 
Since the last member is invariant under the substitution
$(\a,\b,\g,\vp,\psi) \mapsto (\s,\ol\a\vp,\ol\a\psi,\b\s,\g\s)$ and $(\s,\b\s+\g\s)=(\e,\b+\g)=0$,  the theorem follows. 
\end{proof}

By combining Theorem \ref{thomae}, Proposition \ref{shift} and Proposition \ref{reverse}, we obtain many relations among ${}_3F_2(1)$  similarly as in the complex case (cf. \cite{whipple3}). For example, we have the following analogue of Sheppard's formula (cf. loc. cit. p.111, $\dagger$). 

\begin{cor}
Put $\s=\ol{\a\b\g}\vp\psi$ and suppose that $(\a+\s,\e)=(\b+\g,\vp+\psi)=0$.  Then 
$$\FF{\a,\b,\g}{\vp,\psi}{1}= \frac{g(\ol{\b\g}\vp)g^\0(\vp)g(\ol{\b\g}\psi)g^\0(\psi)}{g(\ol\b\vp)g(\ol\g\vp)g(\ol\b\psi)g(\ol\g\psi)}
\FF{\ol\s,\b,\g}{\b\g\ol\vp,\b\g\ol\psi}{1}.$$
\end{cor}

\begin{proof}
First, suppose that $\b \ne \e$. 
By Proposition \ref{reverse} and Proposition \ref{shift}, 
\begin{align*}
\FF{\a,\b,\g}{\vp,\psi}{1}
&=\g(-1)\frac{g^\0(\ol\a)g^\0(\ol\b)g(\g\ol\vp)g(\g\ol\psi)}{g^\0(\ol\a\g)g^\0(\ol\b\g)g(\ol\vp)g(\ol\psi)}
\FF{\g,\g\ol\vp,\g\ol\psi}{\ol\a\g,\ol\b\g}{1}.
\\
 \FF{\ol\s,\b,\g}{\b\g\ol\vp,\b\g\ol\psi}{1}&
 =\g(-1)\frac{g^\0(\s)g^\0(\ol\b)g(\ol\b\vp)g(\ol\b\psi)}{g^\0(\s\g)g^\0(\ol\b\g)g(\ol{\b\g}\vp)g(\ol{\b\g}\psi)}
 \FF{\g,\ol\b\vp,\ol\b\psi}{\g\s,\ol\b\g}{1}. 
\end{align*}
Applying Theorem \ref{thomae} twice, 
\begin{align*}
\frac{g(\g)}{g^\0(\ol\a\g)g^\0(\ol\b\g)} \FF{\g,\g\ol\vp,\g\ol\psi}{\ol\a\g,\ol\b\g}{1}
&=
\frac{g(\s)}{g^\0(\ol{\a\b}\vp)g^\0(\ol{\a\b}\psi)}\FF{\s,\ol\a,\ol\b}{\ol{\a\b}\vp,\ol{\a\b}\psi}{1}
\\&=\frac{g(\s)}{g(\ol\a)} \frac{g(\g)}{g^\0(\g\s)g^\0(\ol\b\g)} 
\FF{\g,\ol\b\vp,\ol\b\psi}{\g\s,\ol\b\g}{1}. 
\end{align*}
Hence the formula follows, using Proposition \ref{jacobi}. The case $\g \ne \e$ is parallel. 
When $\b=\g=\e$, one easily verifies that the both sides of the formula equal
$$q\frac{g(\ol\a)g(\ol\a\vp\psi)}{g^\0(\ol\a\vp)g^\0(\ol\a\psi)}+1,$$
using Theorem \ref{reduction}, Proposition \ref{reverse} and Theorem \ref{summation}.
\end{proof}

The following lemma plays a key role in computing nearly-poised values. 

\begin{lem}\label{nearly}
For any $\vp$, $\b$, $\g$, $\n\in\ck$, 
$$\frac{g^\0(\vp)g(\ol{\b\g}\vp)}{g^\0(\ol\b\vp)g^\0(\ol\g\vp)} 
\frac{(\b+\g)_\n}{(\ol\b\vp+\ol\g\vp)^\0_\n}
=
\frac{1}{1-q}\sum_\m \frac{(\b+\g)_\m}{(\e+\vp)^\0_\m} \frac{(\ol\m)_\n}{(\vp\m)^\0_\n} \n(-1)
+\d c(\b,\g),$$
where $\d=1$ if $\b\g=\vp$ and $\n\in\{\ol\b,\ol\g\}$, $\d=0$ otherwise, and 
$$c(\b,\g):=\frac{1+\d(\ol\b\g)q}{q^{\d(\ol\b\g)}} \frac{(1-q)^2}{q} \frac{g^\0(\b\g)}{g(\b)g(\g)}. $$
\end{lem}

\begin{proof}
First, suppose that $\b\n+\g\n \ne \e+\vp\n^2$. Then, by Theorem \ref{summation},  
$$\FF{\b\n,\g\n}{\vp\n^2}{1}=\frac{g^\0(\vp\n^2)g(\ol{\b\g}\vp)}{g^\0(\ol\b\vp\n)g^\0(\ol\g\vp\n)}
= \frac{g^\0(\vp)g(\ol{\b\g}\vp)}{g^\0(\ol\b\vp)g^\0(\ol\g\vp)}\frac{(\vp)^\0_{\n^2}}{(\ol\b\vp+\ol\g\vp)^\0_\n}.$$
Hence 
\begin{align*}
\frac{g^\0(\vp)g(\ol{\b\g}\vp)}{g^\0(\ol\b\vp)g^\0(\ol\g\vp)}\frac{(\b+\g)_\n}{(\ol\b\vp+\ol\g\vp)^\0_\n} 
& = \frac{(\b+\g)_\n}{(\vp)^\0_{\n^2}} \FF{\b\n,\g\n}{\vp\n^2}{1}
=\frac{1}{1-q} \sum_\r \frac{(\b+\g)_{\n\r}}{(\e)^\0_\r (\vp)^\0_{\n^2\r}}. 
\end{align*}
If we write $\m=\n\r$, then 
$$(\e)^\0_\r=(\e)^\0_\m(\m)^\0_{\ol\n}=\frac{(\e)^\0_\m}{(\ol\m)_\n} \n(-1), \quad 
(\vp)^\0_{\n^2\r} =  (\vp)^\0_\m (\vp\m)^\0_\n,$$
and the formula follows since $\b\g\ne\vp$. 
Secondly, if $\b\n+\g\n=\e+\vp\n^2$ (i.e. $\b\g=\vp$ and $\n\in\{\ol\b,\ol\g\}$), then 
$$\frac{(\b+\g)_\n}{(\vp)^\0_{\n^2}} = q^{-\d(\ol\b\g)} \frac{g^\0(\b\g)}{g(\b)g(\g)},$$
and the formula follows by Remark \ref{r.summation} (i). 
\end{proof}

Recall the following three formulas  (cf. \cite[3.1, 3.3, 3.4]{bailey}).  
Dixon's formula for well-poised ${}_3F_2(1)$: 
\begin{align*}
& \FFF{2a,b,c}{2a-b+1,2a-c+1}{1}
\\&=\frac{\G(2a-b+1)\G(2a-c+1)\G(a+1)\G(a-b-c+1)}{\G(2a+1)\G(2a-b-c+1)\G(a-b+1)\G(a-c+1)}. 
\end{align*}
Watson's formula: 
$$\FFF{2a,2b,c}{a+b+\frac{1}{2},2c}{1}
=\frac{\G(\frac{1}{2})\G(c+\frac{1}{2})\G(a+b+\frac{1}{2})\G(c-a-b+\frac{1}{2})}{\G(a+\frac{1}{2})\G(c-a+\frac{1}{2})\G(b+\frac{1}{2})\G(c-b+\frac{1}{2})}. $$
Whipple's formula: 
if $a+b=d+e-c=1/2$, then
$${}_3F_2\left({2a,2b,c \atop 2d,2e};1\right)
=\frac{\G(d)\G(d+\frac{1}{2})\G(e)\G(e+\frac{1}{2})}{\G(a+d)\G(a+e)\G(b+d)\G(b+e)}.$$
These are all equivalent under Thomae's formula. 

Their finite analogues are as follows 
(cf. \cite[Theorem 4.37, Theorem 4.38 (i), (ii)]{greene} and \cite[Theorem 1.11]{mccarthy} for (i)), and these are all equivalent under Theorem \ref{thomae}. 

\begin{thm}\label{t2.2}
Suppose that $p$ is odd and let $\phi\in\ck$ be the quadratic character.
\begin{enumerate}
\item 
Suppose that $\a^2\ne\b\g$ and $\b+\g\ne \e +\a'$ if $\a'^2=\a^2$. Then 
$$
\FF{\a^2,\b,\g}{\a^2\ol\b,\a^2\ol\g}{1}
=\sum_{\a'^2=\a^2}
\frac{g^\0(\a^2\ol\b)g^\0(\a^2\ol\g)g(\a')g(\a'\ol{\b\g})}
{g(\a^2)g(\a^2\ol{\b\g})g^\0(\a'\ol\b)g^\0(\a'\ol\g)}. $$
\item
Suppose that $(\ol\a\b\phi+\g, \e)=0$ and $(\a^2+\b^2+\g,\e+\a\b\phi+\g^2)\le 1$. Then, 
$$\FF{\a^2,\b^2,\g}{\a\b\phi,\g^2}{1}=\sum_{\n^2=\e}
\frac{g(\phi)g^\0(\g\phi) g^\0(\a\b\phi)g(\ol{\a\b}\g\phi)}{g(\a\n)g^\0(\ol\a\g\n)g(\b\n)g^\0(\ol\b\g\n)}. 
$$
\item
Suppose that $\a\b=\vp\psi\ol\g=\phi$, $(\g,\e+\vp^2+\psi^2)=0$ and 
$(\a^2+\b^2,\e+\vp^2+\psi^2)\le 1$. Then 
$$\FF{\a^2,\b^2,\g}{\vp^2,\psi^2}{1}
= \sum_{\n^2=\e} \frac{g^\0(\vp)g^\0(\vp\phi)g^\0(\psi)g^\0(\psi\phi)}{g^\0(\a\vp\n)g^\0(\a\psi\n)g^\0(\b\vp\n)g^\0(\b\psi\n)}.$$
\end{enumerate}
\end{thm}

\begin{proof}
(i) By Lemma \ref{nearly} and a change of the order of summation, 
\begin{align*}
&\frac{g^\0(\a^2)g(\a^2\ol{\b\g})}{g^\0(\a^2\ol\b)g^\0(\a^2\ol\g)}\FF{\a^2,\b,\g}{\a^2\ol\b,\a^2\ol\g}{1}
= \frac{1}{1-q} \sum_\m \frac{(\b+\g)_\m}{(\e+\a^2)^\0_\m} \FF{\a^2,\ol\m}{\a^2\m}{-1}.
\end{align*}
By Theorem \ref{t2.3} (ii), this equals
\begin{align*}
\frac{1}{1-q}\sum_\m  \frac{(\b+\g)_\m}{(\e+\a^2)^\0_\m }
\sum_{\a'^2=\a^2} q^{\d(\a^2)-\d(\a')} \frac{(\a^2)^\0_\m}{(\a')^\0_\m}
&=\sum_{\a'^2=\a^2} q^{\d(\a^2)-\d(\a')} \FF{\b,\g}{\a'}{1}
\\&=q^{\d(\a^2)} \sum_{\a'^2=\a^2} \frac{g(\a')g(\a'\ol{\b\g})}{g^\0(\a'\ol\b)g^\0(\a'\ol\g)}, 
\end{align*}
where we used Theorem \ref{summation} and the assumption $\b+\g\ne \e+\a'$, hence the result follows.  

(ii) 
By symmetry, it suffices to prove the case where $\a^2 \ne \g^2$ and $\b^2\ne \e$. 
For, otherwise $\b^2\ne \g^2$ and $\a^2\ne\e$ by assumption. 
Applying Theorem \ref{thomae} to the left-hand side of (i) with $\g$ playing the role of $\a$ in loc. cit., 
for which we need that $(\a^2, \b\g+\g^2)=(\a^2+\b,\e)=0$, 
$$
\FF{(\a\ol{\b\g})^2,\a^2\ol{\b\g},(\a\ol\g)^2}{(\a^2\ol{\b\g})^2,\a^2\ol{\b\g^2}}{1}
=
\frac{g^\0((\a^2\ol{\b\g})^2)g^\0(\a^2\ol{\b\g^2})g(\g)}{g((\a\ol{\b\g})^2)g(\a^2)g(\a^2\ol{\b\g})}
\sum_{\a'^2=\a^2}
\frac{g(\a')g(\a'\ol{\b\g})}
{g^\0(\a'\ol\b)g^\0(\a'\ol\g)}. 
$$
Then, replace $\a$, $\b$, $\g$ respectively with $\ol\a\g$, $\ol\a\b\phi$, $\ol{\a\b}\g\phi$. 
The conditions needed for (i) and Theorem \ref{thomae} are satisfied by our assumptions. 
Using the duplication formula, we obtain the result. 

(iii) By symmetry, it suffices to prove the case where $(\a^2,\e)=(\b^2,\vp^2+\psi^2)=0$. 
Then apply Theorem \ref{thomae} to the left-hand side of (ii) with $\g$ playing the role of $\a$ in loc. cit., 
replace $\a$, $\b$ respectively with $\a\vp$, $\a\psi$, and use the duplication formula. 
\end{proof}

Recall Saalsch\"utz's formula (cf. \cite[2.2]{bailey}): 
if $a+b+c+1=d+e$ and one of $a$, $b$, $c$ is a non-positive integer (i.e. the series terminates), then
$$\FFF{a,b,c}{d,e}{1}=\frac{\G(d)\G(1+a-e)\G(1+b-e)\G(1+c-e)}{\G(1-e)\G(d-a)\G(d-b)\G(d-c)}.$$
Its finite analogue is the following (cf. \cite[Theorem 4.35]{greene}).

\begin{thm}\label{saal}
Suppose that $\a\b\g=\vp\psi$ and $\a+\b+\g \ne \e+\vp+\psi$. 
Then 
\begin{align*}
\FF{\a,\b,\g}{\vp,\psi}{1}
&=\frac{g^\0(\vp)g(\a\ol\psi)g(\b\ol\psi)g(\g\ol\psi)}{g(\ol\psi)g^\0(\ol\a\vp)g^\0(\ol\b\vp)g^\0(\ol\g\vp)}
+\frac{g^\0(\vp)g^\0(\psi)}{g(\a)g(\b)g(\g)}. 
\end{align*}
\end{thm}

\begin{proof}
First, suppose that $(\a+\b+\g,\e+\vp+\psi)=0$. 
Then, by Theorem \ref{thomae}, Theorem \ref{reduction}, Proposition \ref{shift} and Theorem \ref{summation}, 
\begin{align*}
\FF{\a,\b,\g}{\vp,\psi}{1} &= G_1 \FF{\e,\ol\a\vp,\ol\a\psi}{\b,\g}{1}
\\&= G_1(qF(\ol\a\vp+\ol\a\psi,\b+\g;1)+1)
\\&= G_1\left(G_2 \FF{\g\ol\vp,\g\ol\psi}{\ol\b\g}{1}+1\right)
\\&=G_1(G_2 G_3+1),
\end{align*}
where 
$$G_1=\frac{g^\0(\vp)g^\0(\psi)}{g(\a)g^\0(\b)g^\0(\g)}, \quad 
G_2=\frac{g^\0(\b)g^\0(\g)g(\g\ol\vp)g(\g\ol\psi)}{g^\0(\ol\b\g)g(\ol\a\vp)g(\ol\a\psi)}, \quad 
G_3=\frac{g^\0(\ol\b\g)g(\a)}{g^\0(\ol\b\vp)g^\0(\ol\b\psi)}, 
$$
hence the formula follows. 
The remaining cases are similarly verified by reducing to Theorem \ref{summation}. 
If $\a=\e$ (then $\b\g=\vp\psi$, $(\b+\g,\vp+\psi)=0$), 
one computes that
$$\FF{\a,\b,\g}{\vp,\psi}{1} =1+ \frac{g^\0(\vp)g^\0(\psi)}{g(\b)g(\g)}.$$
The case $\a=\vp$ (or $\a=\psi$) can be proved similarly (or reduced to the previous case using Proposition \ref{shift}). 
\end{proof}

From Theorem \ref{connection} and its consequences appearing in Corollary \ref{c2.3}, 
one obtains formulas which do not exist over the complex numbers. 
For example, we have the following (cf. \cite[(4.23)--(4.26)]{greene}). 

\begin{cor}
If $(\a+\b,\e+\g)=(\vp+\psi,\e)=0$, then 
$$\FF{\a,\b,\vp}{\g,\vp\psi}{1}
=\frac{g^\0(\g)g(\ol{\a\b}\g)}{g(\ol\a\g)g(\ol\b\g)}\FF{\a,\b,\psi}{\a\b\ol\g,\vp\psi}{1}.$$
\end{cor}
\begin{proof}
Use Theorem \ref{iteration}: multiply the both sides of Theorem \ref{connection} with $\vp(\l)\psi(1-\l)$ and 
take the sums over $\l\in \k$. 
\end{proof}

\subsection{Nearly-poised values}

We have already seen formulas for well-poised values ${}_2F_1(-1)$  (Theorem \ref{t2.3})  and ${}_3F_2(1)$ (Theorem  \ref{t2.2} (i)).  
Recall Whipple's formulas for nearly-poised values ${}_3F_2(-1)$ and ${}_4F_3(1)$ \cite[(2.5), (3.5)]{whipple} (cf. \cite[4.6 (3), 4.5 (1)]{bailey}) 

$$\frac{\G(2k)\G(2k-b-c)}{\G(2k-b)\G(2k-c)}
{}_3F_2\left({2a,b,c \atop 2k-b,2k-c};-1\right)
={}_4F_3\left({k-a,k-a+\frac{1}{2},b,c \atop 2k-2a,k,k+\frac{1}{2}};1\right),$$

\begin{align*}
\frac{\G(2k)\G(2k-a-b)\G(2k-a-c)\G(2k-b-c)}{\G(2k-a)\G(2k-b)\G(2k-c)\G(2k-a-b-c)}
{}_4F_3\left({2s, a,b,c \atop 2k-a,2k-b,2k-c};1\right)
\\
={}_5F_4\left({k-s,k-s+\frac{1}{2},a,b,c \atop 2k-2s,k,k+\frac{1}{2},a+b+c-2k+1};1\right). 
\end{align*}
Note that the last ${}_5F_4(1)$ is Saalsch\"utzian. 

A finite analogue of the first one is the following. 
The reducible case where $\b\g=\vp^2$ will be used in the proof of  Theorem \ref{clausen}.

\begin{thm}\label{t5.1}\  
\begin{enumerate}
\item If $(\a^2,\e+\vp^2)=0$, then 
\begin{align*}
&\frac{g^\0(\vp^2)g(\ol{\b\g}\vp^2)}{g^\0(\ol\b\vp^2)g^\0(\ol\g\vp^2)}\FF{\a^2,\b,\g}{\ol\b\vp^2,\ol\g\vp^2}{-1} 
\\& =   \FF{\ol\a\vp, \ol\a\vp\phi, \b,\g}{\ol\a^2\vp^2,\vp ,\vp\phi}{1}
+\d(\ol{\b\g}\vp^2) \frac{c(\b,\g)}{1-q}\sum_{\n\in\{\ol\b,\ol\g\}} \frac{(\a^2)_\n}{(\e)^\0_\n}\n(-1),
\end{align*}
where $c(\b,\g)$ is as in Lemma \ref{nearly}.
\item If $\vp^2\ne\e$, then 
\begin{align*}
&\frac{g^\0(\vp^2)g(\ol{\b\g}\vp^2)}{g^\0(\ol\b\vp^2)g^\0(\ol\g\vp^2)}
\FF{\vp^2,\b,\g}{\ol\b\vp^2,\ol\g\vp^2}{-1} 
\\&=   \FF{\phi, \b,\g}{\vp ,\vp\phi}{1} + 1 + \d(\ol{\b\g}\vp^2) (2-\d(\ol\b\g)) \frac{1-q^{1+\d(\ol\b\g)}}{q^{1+\d(\ol\b\g)}}. 
\end{align*}
\end{enumerate}
\end{thm}

\begin{proof}By Lemma \ref{nearly} and an exchange of the order of summation, 
\begin{align*}
&\frac{g^\0(\vp^2)g(\ol{\b\g}\vp^2)}{g^\0(\ol\b\vp^2)g^\0(\ol\g\vp^2)} \FF{\a^2,\b,\g}{\ol\b\vp^2,\ol\g\vp^2}{-1} 
\\&= \frac{1}{1-q}\sum_\m \frac{(\b+\g)_\m}{(\e+\vp^2)^\0_\m} \FF{\a^2,\ol\m}{\vp^2\m}{1} 
+\d(\ol{\b\g}\vp^2) \frac{c(\b,\g)}{1-q}\sum_{\n\in\{\ol\b,\ol\g\}} \frac{(\a^2)_\n}{(\e)^\0_\n}\n(-1).
\end{align*}
By Theorem \ref{summation} (see Remark \ref{r.summation}) and the duplication formula, 
if $\a^2\ne\e$, then  
\begin{align*}
\FF{\a^2,\ol\m}{\vp^2\m}{1}
& =  q^{-\d(\ol\a^2\vp^2)} \frac{(\vp^2)^\0_\m(\ol\a^2\vp^2)_{\m^2}}{(\ol\a^2\vp^2)^\0_\m(\vp^2)^\0_{\m^2}} 
-\d(\m)\d(\ol\a^2\vp^2) \frac{(1-q)^2}{q}
\\& = q^{-\d(\ol\a^2\vp^2)}   \frac{(\vp^2)^\0_\m(\ol\a\vp+\ol\a\vp\phi)_\m}{(\ol\a^2\vp^2+\vp+\vp\phi)^\0_\m}
 -\d(\m)\d(\ol\a^2\vp^2) \frac{(1-q)^2}{q}. 
\end{align*}
Now the formula (i) follows immediately. 
As for the case (ii) where $\a^2=\vp^2$, we have by Example \ref{e3.1} (i) 
$$\FF{\ol\a\vp, \ol\a\vp\phi, \b,\g}{\ol\a^2\vp^2,\vp ,\vp\phi}{1}
=\FF{\e,\phi,\b,\g}{\e,\vp,\vp\phi}{1}
=q\FF{\phi,\b,\g}{\vp,\vp\phi}{1}+1,$$
and the rest is easy. 
\end{proof}

\begin{cor}
Suppose that $(\a^2,\e+\vp^2)=(\b,\vp)=0$. Then,  
\begin{align*}
\frac{g^\0(\vp^2)g(\ol\b\vp)}{g^\0(\ol\b\vp^2)g^\0(\vp)} 
\FF{\a^2,\b}{\ol\b\vp^2}{-1}
=
\FF{\ol\a\vp,\ol\a\vp\phi,\b}{\ol\a^2\vp^2,\vp\phi}{1}.
\end{align*}
\end{cor}

\begin{proof}
Set $\g=\vp$ in Theorem \ref{t5.1} (i) and apply Theorem \ref{reduction} to the both sides. 
Then, use the duplication formula. 
\end{proof}

\begin{rmk}
A similar formula for the case where $\a^2=\vp^2$ (resp. $\b=\vp$) reduces to Theorem \ref{t2.3} (resp. Theorem \ref{saal}). 
\end{rmk}

A finite analogue of the second formula of Whipple mentioned above is the following. 

\begin{thm}
Suppose that $(\a\b+\a\g+\b\g,\vp^2)=0$. 
\begin{enumerate}
\item
If $(\s^2,\e+\vp^2)=0$, then 
\begin{align*}
&\frac{g(\ol{\a\b}\vp^2)g(\ol{\a\g}\vp^2)g(\ol{\b\g}\vp^2)}
{g^\0(\ol\a\vp^2)g^\0(\ol\b\vp^2)g^\0(\ol\g\vp^2)}
\FF{\s^2,\a,\b,\g}{\ol\a\vp^2,\ol\b\vp^2,\ol\g\vp^2}{1}
\\&=\frac{g(\ol{\a\b\g}\vp^2)}{g^\0(\vp^2)} 
\FF{\ol\s\vp,\ol\s\vp\phi,\a,\b,\g}{\ol\s^2\vp^2,\vp,\vp\phi,\a\b\g\ol\vp^2}{1}
-\frac{g^\0(\ol\a)g^\0(\ol\b)g^\0(\ol\g)}{q^2\s(4)}.
\end{align*}

\item
If $\vp^2\ne \e$, then 
\begin{align*}
&
\frac{g(\ol{\a\b}\vp^2)g(\ol{\a\g}\vp^2)g(\ol{\b\g}\vp^2)}
{g^\0(\ol\a\vp^2)g^\0(\ol\b\vp^2)g^\0(\ol\g\vp^2)}
\FF{\vp^2,\a,\b,\g}{\ol\a\vp^2,\ol\b\vp^2,\ol\g\vp^2}{1}
\\& = 
\frac{g(\ol{\a\b\g}\vp^2)}{g^\0(\vp^2)} \left(
\FF{\phi,\a,\b,\g}{\vp,\vp\phi,\a\b\g\ol\vp^2}{1} 
+1\right)
-\frac{g^\0(\ol\a)g^\0(\ol\b)g^\0(\ol\g)}{q^2\vp(4)}.
\end{align*}
\end{enumerate}
\end{thm}

\begin{proof}
If $\a^2=\b^2=\g^2=\vp^2$, then $(\a\b+\a\g+\b\g,\vp^2)\ne 0$. 
Therefore, by symmetry we can assume that $\a^2\ne\vp^2$. 
Similarly as in the proof of Theorem \ref{t5.1}, since $\b\g \ne \vp^2$, 
\begin{align*}
&\frac{g^\0(\vp^2)g(\ol{\b\g}\vp^2)}{g^\0(\ol\b\vp^2)g^\0(\ol\g\vp^2)}
\FF{\s^2,\a,\b,\g}{\ol\a\vp^2,\ol\b\vp^2,\ol\g\vp^2}{1}
= \frac{1}{1-q}  \sum_\m \frac{(\b+\g)_\m}{(\e+\vp^2)^\0_\m}
A(\m), 
\end{align*}
where  
$$A(\m):=\FF{\s^2,\a,\ol\m}{\ol\a\vp^2, \m\vp^2}{-1}.$$

(i) Suppose that $(\s^2,\e+\vp^2)=0$. By Theorem \ref{t5.1} (i), 
$$
\frac{(\ol\a\vp^2)^\0_\m}{(\vp^2)^\0_\m} A(\m)=
\begin{cases}
B(\m) & (\m \ne \a\ol\vp^2), 
\\
q B(\a\ol\vp^2) + qC & (\m=\a\ol\vp^2),
\end{cases}
$$
where 
\begin{align*}
B(\m)
&:=\FF{\ol\s\vp,\ol\s\vp\phi,\a,\ol\m}{\ol\s^2\vp^2,\vp,\vp\phi}{1}
\\
C&:=\frac{c(\a,\ol\a\vp^2)}{1-q} \sum_{\n\in\{\ol\a,\a\ol\vp^2\}} \frac{(\s^2)_\n}{(\e)^\0_\n} \n(-1)
=\frac{1-q}{q} \frac{g^\0(\vp^2)}{g(\s^2)}
\left(\frac{g(\ol\a\s^2)}{g(\ol\a\vp^2)}+\frac{g(\a\ol\vp^2\s^2)}{g(\a)}\right).  
\end{align*}
Note that $\ol\a \ne \a\ol\vp^2$ by assumption.  
Therefore, 
\begin{align*}
&\frac{1}{1-q}  \sum_\m \frac{(\b+\g)_\m}{(\e+\vp^2)^\0_\m} A(\m)
\\&=
\frac{1}{1-q}  \sum_\m \frac{(\b+\g)_\m}{(\e+\ol\a\vp^2)^\0_\m} B(\m)
-\frac{1}{q} \frac{(\b+\g)_{\a\ol\vp^2}}{(\e+\vp^2)^\0_{\a\ol\vp^2}} A(\a\ol\vp^2)
+\frac{1}{1-q} \frac{(\b+\g)_{\a\ol\vp^2}}{(\e+\ol\a\vp^2)^\0_{\a\ol\vp^2}}C. 
\end{align*}
First, by a change of the order of summation, using $(\ol\m)_\n=(\e)_\n(\e)^\0_\m/(\ol\n)^\0_\m$, 
\begin{align*}
&\frac{1}{1-q}  \sum_\m \frac{(\b+\g)_\m}{(\e+\ol\a\vp^2)^\0_\m} B(\m)
\\&=\frac{1}{1-q} \sum_\n \frac{(\ol\s\vp+\ol\s\vp\phi+\a+\e)_\n}{(\e+\ol\s^2\vp^2+\vp+\vp\phi)^\0_\n} 
F(\b+\g,\ol\a\vp^2+\ol\n;1)
\\&
=\frac{1}{1-q} \sum_\n \frac{(\ol\s\vp+\ol\s\vp\phi+\a+\e)_\n}{(\e+\ol\s^2\vp^2+\vp+\vp\phi)^\0_\n} \frac{(\b+\g)_\n}{(\ol\a\vp^2+\ol\n)^\0_\n} 
\FF{\b\n,\g\n}{\ol\a\vp^2\n}{1}
\\& 
=\frac{1}{1-q} \sum_\n \frac{(\ol\s\vp+\ol\s\vp\phi+\a+\e)_\n}{(\e+\ol\s^2\vp^2+\vp+\vp\phi)^\0_\n} 
 \frac{g^\0(\ol\a\vp^2)g(\ol{\a\b\g}\vp^2)}{g(\ol{\a\b}\vp^2)g(\ol{\a\g}\vp^2)} \frac{(\b+\g)_\n}{(\e)_\n(\a\b\g\ol\vp^2)^\0_\n}
\\&
=\frac{g^\0(\ol\a\vp^2)g(\ol{\a\b\g}\vp^2)}{g(\ol{\a\b}\vp^2)g(\ol{\a\g}\vp^2)} 
\FF{\ol\s\vp,\ol\s\vp\phi,\a,\b,\g}{\ol\s^2\vp^2,\vp,\vp\phi,\a\b\g\ol\vp^2}{1}. 
\end{align*}
Here we used Proposition \ref{shift} and Theorem \ref{summation} together with the assumption $(\a\b+\a\g,\vp^2)=0$. 
Secondly, by Theorem \ref{reduction}, 
\begin{align*}
& A(\a\ol\vp^2)=\FF{\s^2,\a,\ol\a\vp^2}{\a,\ol\a\vp^2}{-1}
\\&=q^{\d(\a)+\d(\ol\a\vp^2)}\ol\s(4) 
+q^{\d(\ol\a\vp^2)-1} 
\frac{g^\0(\a)g(\ol\a\s^2)}{g(\s^2)}
+q^{\d(\a)-1} \frac{g^\0(\ol\a\vp^2)g(\a\ol\vp^2\s^2)}{g(\s^2)}.
\end{align*}
Putting all together, we obtain the formula. 

(ii) Let $\s^2=\vp^2\ne\e$. Then by Theorem \ref{t5.1} (ii), we have
\begin{align*}
A(\m)=\frac{(\vp^2)^\0_\m}{(\ol\a\vp^2)^\0_\m} (B(\m)+1) +\d(\ol\a\vp^2\m)(q-1) 
(B(\a\ol\vp^2)-1),
\end{align*}
where 
$$B(\m):=\FF{\phi,\a,\ol\m}{\vp,\vp\phi}{1}.$$
Therefore, as above, 
\begin{align*}
\frac{1}{1-q}  \sum_\m \frac{(\b+\g)_\m}{(\e+\vp^2)^\0_\m}A(\m)
=& \frac{g^\0(\ol\a\vp^2)g(\ol{\a\b\g}\vp^2)}{g(\ol{\a\b}\vp^2)g(\ol{\a\g}\vp^2)}  \FF{\phi,\a,\b,\g}{\vp,\vp\phi,\a\b\g\ol\vp^2}{1} 
\\& + \FF{\b,\g}{\ol\a\vp^2}{1}
- \frac{(\vp^2)^\0_{\a\ol\vp^2}}{(\ol\a\vp^2)^\0_{\a\ol\vp^2}}(B(\a\ol\vp^2)-1).
\end{align*}
Applying Theorem \ref{summation} to the second term and Theorem \ref{saal} to the last term,  the result follows. 
\end{proof}

\begin{rmk}
Finite analogues of Whipple's formulas for well-poised values ${}_4F_3(-1)$ and ${}_5F_4(1)$ are given by McCarthy \cite[Theorems 1.5, 1.6]{mccarthy}. 
\end{rmk}

\section{Quadratic transformation formulas}\label{s5}

Many transformation formulas are known for complex hypergeometric functions (see for example \cite{otsubo2} and its references). 
Here we prove finite analogues of some quadratic transformation formulas and their consequences.  
Differential equation, the most powerful tool in proving complex formulas, is no longer available here. 
Instead, we compare the Fourier transforms of functions in question.

\subsection{Transformations of ${}_2F_1(\l)$}

In this section, we discuss quadratic formulas and some resulting quartic formulas. 
Throughout this section, we assume that $p$ is odd, and $\phi\in\ck$ denotes the quadratic character. 

First, recall transformation formulas respectively of Gauss \footnote{Though Ramanujan is referred to in \cite[Section 4]{otsubo2}, it was already known by Gauss \cite[Formula 100]{gauss}. } and Kummer (cf. \cite[(4.2), (4.1)]{otsubo2})
\begin{align*}
(1+x)^{2a} {}_2F_1\left({2a,b \atop 2a-b+1};x\right)&=
{}_2F_1\left({a,a+\frac{1}{2} \atop 2a-b+1};1-\left(\frac{1-x}{1+x}\right)^2\right), 
\\
(1+x)^{2a} {}_2F_1\left({a,a+\frac{1}{2} \atop b+\frac{1}{2}};x^2\right)&=
{}_2F_1\left({2a,b\atop 2b};1-\frac{1-x}{1+x}\right).
\end{align*}
From the viewpoint of differential equations, these are equivalent to each other (see loc. cit.). 
Their finite analogues are the following (cf. \cite[Theorem 4.20]{greene}). 

\begin{thm}\label{quad1}
Suppose that $(\a^2+\b,\e)=0$. 
\begin{enumerate}
\item
If $\l \ne -1$,  then 
$$\a^2(1+\l) \FF{\a^2,\b}{\a^2\ol\b}{\l}=\FF{\a,\a\phi}{\a^2\ol\b}{1-\left(\frac{1-\l}{1+\l}\right)^2}.$$
\item If $\l \ne -1$, then 
$$\a^2(1+\l) \FF{\a,\a\phi}{\b\phi}{\l^2}=\FF{\a^2,\b}{\b^2}{1-\frac{1-\l}{1+\l}}.$$
\end{enumerate}\end{thm}

\begin{proof}
(i) Put
$$f(\l)=\ol\a^2(1+\l)\FF{\a,\a\phi}{\a^2\ol\b}{1-\left(\frac{1-\l}{1+\l}\right)^2}$$
and extend this to $\k^*$ by setting $f(-1)=0$. 
Then for any $\m\in\ck$, 
\begin{align*}
\wh{f}(\m)
&=\frac{1}{1-q}\sum_\n  \frac{(\a+\a\phi)_\n}{(\e+\a^2\ol\b)^\0_\n}
\sum_\l \n(4) \ol\m\n(\l)\ol{\a^2\n^2}(1+\l)
\\& = -\frac{1}{1-q} \sum_\n  \frac{(\a+\a\phi)_\n}{(\e+\a^2\ol\b)^\0_\n}
j(\ol\m\n,\ol{\a^2\n^2}) \n(4)\m\n(-1).
\end{align*}
Unless $\m=\n=\ol\a'$ with $\a'^2=\a^2$, we have by the duplication formula, 
$$j(\ol\m\n,\ol{\a^2\n^2}) \n(4)\m\n(-1)
=\frac{g(\ol\m\n)g(\a^2\m\n)}{g^\0(\a^2\n^2)} \n(4)
=\m(-1)\frac{(\a^2)_\m}{(\e)^\0_\m} \frac{(\ol\m+\a^2\m)_\n}{(\a+\a\phi)^\0_\n}.$$
On the other hand, if $\m=\n=\ol\a'$ with $\a'^2=\a^2$, then since 
$$j(\e,\e)=\frac{g(\e)^2}{g^\0(\e)}-\frac{(1-q)^2}{q},$$
we have 
$$j(\ol\m\n,\ol{\a^2\n^2}) \n(4)\m\n(-1)
=\m(-1)\frac{(\a^2)_\m}{(\e)^\0_\m} \frac{(\ol\m+\a^2\m)_\n}{(\a+\a\phi)^\0_\n} -\frac{(1-q)^2}{q}  \ol{\a'}(4).$$
Hence, 
$$\wh{f}(\m)
=-\m(-1)\frac{(\a^2)_\m}{(\e)^\0_\m} \FF{\ol\m,\a^2\m,\a,\a\phi}{\a^2\ol\b,\a,\a\phi}{1}  
+\d(\a^2\m^2) \m(-1) \frac{1-q}{q} \frac{g^\0(\a^2\ol\b)g(\ol\m)}{g(\a^2)g^\0(\ol{\m\b})}. $$
By Theorem \ref{reduction} and Theorem \ref{summation}, 
\begin{align*}
\FF{\ol\m,\a^2\m,\a,\a\phi}{\a^2\ol\b,\a,\a\phi}{1}
&=\FF{\ol\m,\a^2\m}{\a^2\ol\b}{1} 
+\sum_{\a'^2=\a^2} q^{-1}\frac{(\ol\m+\a^2\m)_{\ol{\a'}}}{(\e+\a^2\ol\b)^\0_{\ol{\a'}}}
\\&=\frac{g^\0(\a^2\ol\b)g(\ol\b)}{g^\0(\a^2\ol\b\m)g^\0(\ol{\b\m})}
+\sum_{\a'^2=\a^2} q^{-1}\frac{g^\0(\m)g(\a'\m)g^\0(\a^2\ol\b)g(\a')}{g(\a^2\m)g^\0(\a'\m)g^\0(\a'\ol\b)}
\\&=\m(-1) \frac{(\b)_\m}{(\a^2\ol\b)^\0_\m}
+\frac{(\e)^\0_\m}{(\a^2)_\m} \sum_{\a'^2=\a^2} q^{-\d(\a'\m)}\frac{g^\0(\a^2\ol\b)g(\a')}{g(\a^2)g^\0(\a'\ol\b)}. 
\end{align*}
It follows by Theorem \ref{t2.3} (ii) that
\begin{align*}
\wh{f}(\m)
=-\frac{(\a^2+\b)_\m}{(\e+\a^2\ol\b)^\0_\m}
-\m(-1) \FF{\a^2,\b}{\a^2\ol\b}{-1}
\end{align*}
for any $\m\in\ck$. Therefore, 
$$f(\l)= \FF{\a^2,\b}{\a^2\ol\b}{\l} - \d(1+\l) \FF{\a^2,\b}{\a^2\ol\b}{-1}$$
for any $\l\in\k^*$ (recall Example \ref{e2.2} (i)), hence the formula.  

(ii) If $\a^2\ne\b^2$, apply Theorem \ref{connection} to the both sides of (i), replace $\l$ with $\frac{1-\l}{1+\l}$ and use the duplication formula. 
If $\a^2=\b^2$ and $\l\ne 0$, then 
the both sides equal 
$$\ol\b\left(\frac{1-\l}{1+\l}\right) + \frac{g(\ol\b)}{g(\b)g(\ol\b^2)} \ol\b^2\left(\frac{2\l}{1+\l}\right)$$ 
by Theorem \ref{reduction}, Corollary \ref{c2.0} and the duplication formula. 
\end{proof}

\begin{rmk}
Theorem \ref{quad1} (i) can also be proved as follows. 
By Lemma \ref{nearly} and a change of the order of summation, 
\begin{align*}
\frac{g^\0(\a^2)g(\ol\b)}{qg^\0(\a^2\ol\b)} 
\FF{\a^2,\b}{\a^2\ol\b}{\l}=\frac{1}{1-q}\sum_\m \frac{(\a^2+\b)_\m}{(\e+\a^2)^\0_\m} F(\ol\m,\a^2\m;-\l). 
\end{align*}
By Proposition \ref{shift} and Corollary \ref{c2.0}, 
\begin{align*}
F(\ol\m,\a^2\m;-\l) = \a^2\left(\frac{1+\l}{\l}\right)  \frac{(\e+\a^2)^\0_\m}{(\a+\a\phi)^\0_\m} \m\left(\frac{(1+\l)^2}{4\l}\right). 
\end{align*}
Then, by Proposition \ref{reverse}, Proposition \ref{shift} and the duplication formula, 
\begin{align*}
& F\left(\a^2+\b, \a+\a\phi; \frac{(1+\l)^2}{4\l}\right)
\\& =F\left(\ol\a+\ol\a\phi, \ol\a^2+\ol\b; \frac{4\l}{(1+\l)^2}\right)
\\& =\frac{g^\0(\a^2)g(\ol\b)}{qg^\0(\a^2\ol\b)} 
\a^2\left(\frac{\l}{(1+\l)^2}\right) \FF{\a,\a\phi}{\a^2\b}{\frac{4\l}{(1+\l)^2}}. 
\end{align*}
Hence the formula follows. 
\end{rmk}

Recall the formulas respectively of Gauss and Kummer (cf. \cite[(4.5), (4.6)]{otsubo2})
\begin{align*}
{}_2F_1\left({2a,2b \atop a+b+\frac{1}{2}};x\right)
&={}_2F_1\left({a,b\atop a+b+\frac{1}{2}};1-(1-2x)^2\right), 
\\
(1+x)^{2a} {}_2F_1\left({2a,a-b+\frac{1}{2} \atop a+b+\frac{1}{2}};-x\right)
&={}_2F_1\left({a,b\atop a+b+\frac{1}{2}};1-\left(\frac{1-x}{1+x}\right)^2\right).
\end{align*}
Their finite analogues are the following (for (ii), cf. \cite[Theorem 9.4]{fuselieretal}). 

\begin{cor}\label{quad2}
Suppose that $(\a^2+\b^2+\a\ol\b\phi,\e)=0$. 
\begin{enumerate}
\item  If $\l \ne 1, 1/2$, then 
$$\FF{\a^2,\b^2}{\a\b\phi}{\l}
= \FF{\a,\b}{\a\b\phi}{1-(1-2\l)^2}.$$
\item If $\l \ne \pm 1$, then 
\begin{align*}
\a^2(1+\l) \FF{\a^2,\a\ol\b\phi}{\a\b\phi}{-\l}=\FF{\a,\b}{\a\b\phi}{1-\left(\frac{1-\l}{1+\l}\right)^2}.
\end{align*}
\end{enumerate}
\end{cor}
\begin{proof}
(ii) In Theorem \ref{quad1} (i), replace $\l$ with $-\l$, $\b$ with $\a \ol\b \phi$, and apply Theorem \ref{p2.5} (ii) to the right-hand side. 

(i) In (ii), replace $\l$ with $\frac{\l}{1-\l}$ and apply Theorem \ref{p2.5} (ii) to the left-hand side. 
\end{proof}

Recall the formula of Gauss (cf. \cite[(1.1)]{otsubo2})
\begin{align*}
(1+x)^{2a}{}_2F_1\left({a,a-b+\frac{1}{2} \atop b+\frac{1}{2}};{x^2}\right)&=
{}_2F_1\left({a,b \atop 2b};{1-\left(\frac{1-x}{1+x}\right)^2}\right). 
\end{align*}
Its finite analogue is the following (cf. {\cite[Theorem 2]{evans}}). 
\begin{thm}\label{quad5}
Suppose that $(\a,\e+\b\phi+\b^2)=(\b,\e)=0$. 
Then for $\l \ne -1$, 
$$\a^2(1+\l)\FF{\a,\a\ol\b\phi}{\b\phi}{\l^2}=\FF{\a,\b}{\b^2}{1-\left(\frac{1-\l}{1+\l}\right)^2}.$$
\end{thm}

\begin{proof}
The proof is similar to the proof of Theorem \ref{quad1} (i). 
Put a function on $\k^*\setminus\{-1\}$ as
$$f(\l)=\ol\a^2(1+\l) \FF{\a,\b}{\b^2}{1-\left(\frac{1-\l}{1+\l}\right)^2}.$$ 
Then by Theorem \ref{summation}, 
$$f(1)=\ol\a(4)\frac{g^\0(\b^2)g(\ol\a\b)}{g^\0(\ol\a\b^2)g(\b)}.$$
By Corollary \ref{c2.3}, one sees easily that $f$ is even on $\k^*\setminus\{\pm 1\}$. 
Extend $f$ to an even function on $\k^*$, i.e. 
$$f(\l)=\frac{1}{1-q}\sum_\n \frac{(\a+\b)_\n}{(\e+\b^2)^\0_\n}\n(4\l)\ol{\a^2\n^2}(1+\l) +\d(1+\l) f(1).$$
We are to show 
$$\wh f(\m^2) = -\sum_{\n^2=\m^2}\frac{(\a+\a\ol\b\phi)_\n}{(\e+\b\phi)^\0_\n}$$ for all $\m\in\ck$.  
Similarly as before,  we have
\begin{align*}
& \wh f(\m^2)\\
& = -q^{-\d(\a\phi)} \frac{(\a^2)_{\m^2}}{(\e)^\0_{\m^2}}
\FF{\a,\b,\ol\m^2,\a^2\m^2}{\b^2,\a,\a\phi}{1} +\d(\a^2\m^4) \frac{1-q}{q} \ol\a(4)\frac{(\a+\b)_{\m^2}}{(\e+\b^2)^\0_{\m^2}}  +f(1)
\\&= 
-q^{-\d(\a\phi)} \frac{(\a^2)_{\m^2}}{(\e)^\0_{\m^2}}
\FF{\b,\ol\m^2,\a^2\m^2}{\b^2,\a\phi}{1} 
+ \d(\a\phi\m^2)\frac{1-q}{q} 
\frac{(\a^2)_{\ol\a\phi}}{(\e)^\0_{\ol\a\phi}}
\frac{g^\0(\b^2)g(\ol\a\b\phi)}{g(\b)g^\0(\ol\a\b^2\phi)}.
\end{align*}
First, if $(\m^2,\e+\ol\a^2+\ol\a\phi)=0$, then by Theorem \ref{t2.2} (ii),  
$$\FF{\b,\ol\m^2,\a^2\m^2}{\b^2,\a\phi}{1} 
= q^{\d(\a\phi)} \sum_{\n^2=\m^2} \frac{(\phi)^\0_\n(\a\ol\b\phi)_\n}{(\a\phi)_\n(\b\phi)^\0_\n}.$$
This equality is also valid if $\m^2=\e\ne\ol\a\phi$ (resp. $\m^2=\ol\a^2\ne\ol\a\phi$), 
as both sides coincide with 
$$q\frac{g(\a\ol\b)g(\a\phi)g^\0(\b\phi)}{g(\phi)g(\a)g(\b)g(\a\ol\b\phi)}+1, \quad 
\left(\text{resp. } \ \frac{g(\a\phi)g^\0(\b\phi)g(\ol\a\b\phi)g(\phi)}{g(\a)g(\b)g^\0(\ol\a\b)}+1\right)$$
(use Theorem \ref{reduction}, Proposition \ref{shift} and Theorem \ref{summation} for the left member).   
On the other hand, if $\m^2=\ol\a\phi$, we have similarly  
$$\FF{\b,\ol\m^2,\a^2\m^2}{\b^2,\a\phi}{1} 
= q^{\d(\a\phi)}\frac{1+q}{q} \frac{g^\0(\b^2)g(\ol\a\b\phi)}{g(\b)g^\0(\ol\a\b^2\phi)},$$
and also, 
$$q^{\d(\a\phi)} \sum_{\n^2=\m^2} \frac{(\phi)^\0_\n(\a\ol\b\phi)_\n}{(\a\phi)_\n(\b\phi)^\0_\n}
=2q^{\d(\a\phi)} \frac{g^\0(\b^2)g(\ol\a\b\phi)}{g(\b)g^\0(\ol\a\b^2\phi)}.$$
In any case, 
$$\wh f(\m^2)= -\frac{(\a^2)_{\m^2}}{(\e)^\0_{\m^2}} \sum_{\n^2=\m^2} \frac{(\phi)^\0_\n(\a\ol\b\phi)_\n}{(\a\phi)_\n(\b\phi)^\0_\n}
=-\sum_{\n^2=\m^2} 
\frac{(\a+\a\ol\b\phi)_\n}{(\e+\b\phi)^\0_\n}
$$
as we wanted. 
\end{proof}

Recall the formula of Ramanujan--Matsumoto--Ohara (cf. \cite[Section 1]{otsubo2})
$$(1+3x)^{6a}{}_2F_1\left({3a,3a+\frac{1}{2} \atop 2a+\frac{5}{6}};{x^2}\right)= 
{}_2F_1\left({3a,3a+\frac{1}{2} \atop 4a+\frac{2}{3}};{1-\left(\frac{1-x}{1+3x}\right)^2}\right).$$
Its finite analogue is the following. 

\begin{thm}
\ 
Suppose that $3 \mid q-1$ and let $\r$ be a cubic character. 
Suppose that $\a^6\ne \e$. Then for $\l\ne -1$, $-1/3$,  
$$\a^6(1+3\l) \FF{\a^3,\a^3\phi}{\a^2\phi\r}{\l^2}=\FF{\a^3,\a^3\phi}{\a^4\r^2}{1-\left(\frac{1-\l}{1+3\l}\right)^2}.$$
\end{thm}

\begin{proof}
In Theorem \ref{quad1} (i) and (ii), replace $\a$, $\b$ with  $\a^3$, $\a^2\r$ respectively and in the former, replace $\l$ with $1-\frac{1-\l}{1+\l}$. Then compare the resulting formulas.  
\end{proof}

The proof as above imitates the derivation of the complex analogue in \cite[Section 4]{otsubo2}.
If $\a$ is a square, Theorem \ref{quad5} can also be derived from Theorem \ref{quad1} as in loc. cit.  
The proof of the following corollary imitates the proof of \cite[Corollary 6.1, Remark 6.3]{otsubo2}.

\begin{cor}\label{c5.2}\ 
\begin{enumerate}
\item
Suppose that $3 \mid q-1$ and let $\r$ be a cubic character. 
If $(\a^3,\e)=(\a^2,\phi\r)=0$, $\l \ne -1$ and $\l^2\ne -1$, then 
\begin{align*}
\a^{12}(1+\l) \FF{\a^3,\a^2\phi\ol\r}{\a\phi\r}{\l^4}&
=\FF{\a^3,\a^2\phi\ol\r}{\a^4\r}{1-\left(\frac{1-\l}{1+\l}\right)^4}.
\end{align*}
\item
Suppose that $4 \mid q-1$ and let $\s$ be a quartic character. 
If $(\a^2+\a\s,\e)=0$ and $\l^4\ne 1$, then
\begin{align*}
\a^4(1+\l) \FF{\a^2,\a\s}{\a\ol\s}{-\l^2}&
=\FF{\a,\a\s}{\a^2\phi}{1-\left(\frac{1-\l}{1+\l}\right)^4}.
\end{align*}
\end{enumerate}
\end{cor}

\begin{proof}
(i) In Theorem \ref{quad5}, replace $\a$ with $\a^3$ and $\b$ with $\a\r$, and set $\l=x^2$. 
On the other hand, in loc. cit., replace $\a$ with $\a^3$ and $\b$ with $\a^2\phi\ol\r$, and set $\l= \frac{2x}{1+x^2}$. Then, compare the resulting formulas. 

(ii) In Corollary \ref{quad2} (ii), set $\b=\s$ and $\l=x^2$. 
On the other hand, in Theorem \ref{quad5},  set $\b=\a\s$ and $\l=\frac{2x}{1+x^2}$. Then, compare the resulting formulas. 
\end{proof}

\begin{rmk}
Corollary \ref{c5.2} (ii) is equivalent to \cite[Theorem 3]{evans} by Theorem \ref{connection}. 
\end{rmk}

\subsection{Transformation of ${}_3F_2$}

Recall Whipple's ${}_3F_2$ quadratic transformation formula (cf. \cite[4.5. (1)]{erdelyi}) 
\begin{align*}
 &{}_3F_2\left({2a,b,c \atop 2a-b+1,2a-c+1};-x\right)
 \\& =(1+x)^{-2a}
 {}_3F_2\left({a,a+\frac{1}{2},2a-b-c+1 \atop 2a-b+1,2a-c+1}; 1-\left(\frac{1-x}{1+x}\right)^2\right). 
\end{align*}
Its finite analogue is the following (cf. {\cite[Corollary 4.30]{greene}}). 

\begin{thm}\label{quad4}
Suppose that $(\a^2+\b+\g,\e)=(\a^2,\b\g)=0$. Then for any $\l\ne-1$, 
\begin{align*}
&\FF{\a^2,\b,\g}{\a^2\ol\b,\a^2\ol\g}{-\l} 
-\d(1-\l) \frac{g^\0(\a^2\ol\b)g^\0(\a^2\ol\g)}{g(\a^2)g(\a^2\ol{\b\g})}
\\&= \ol\a^2(1+\l) 
\FF{\a,\a\phi,\a^2\ol{\b\g}}{\a^2\ol\b,\a^2\ol\g}{1-\left(\frac{1-\l}{1+\l}\right)^2}. \end{align*}
In particular, 
$$\FF{\a^2,\b,\g}{\a^2\ol\b,\a^2\ol\g}{-1}-\frac{g^\0(\a^2\ol\b)g^\0(\a^2\ol\g)}{g(\a^2)g(\a^2\ol{\b\g})}
=\ol\a(4) \FF{\a,\a\phi,\a^2\ol{\b\g}}{\a^2\ol\b,\a^2\ol\g}{1}.$$
\end{thm}

\begin{proof}The proof is again  similar to the proof of Theorem \ref{quad1} (i). 
Put 
$$f(\l)= \ol\a^2(1+\l)\cdot  \frac{1}{1-q}
\sum_\n \frac{(\a+\a\phi+\a^2\ol{\b\g})_\n}{(\e+\a^2\ol\b+\a^2\ol\g)^\0_\n}\n(4\l)\ol\n^2(1+\l).
$$
Then one computes using Theorem \ref{reduction} 
\begin{align*}
\wh f(\m)
=& -\m(-1) \frac{(\a^2)_\m}{(\e)^\0_\m} \FF{\a^2\ol{\b\g},\ol\m,\a^2\m,\a,\a\phi}{\a^2\ol\b,\a^2\ol\g,\a,\a\phi}{1}
\\&+ \d(\a^2\m^2)\frac{1-q}{q}C \m(-1)  
\frac{g(\ol\m)g(\ol{\m\b\g})}{g^\0(\ol{\m\b})g^\0(\ol{\m\g})}
\\=& 
-\m(-1) \frac{(\a^2)_\m}{(\e)^\0_\m} \FF{\a^2\ol{\b\g},\ol\m,\a^2\m}{\a^2\ol\b,\a^2\ol\g}{1}
-C \m(-1) \sum_{\a'^2=\a^2}  \frac{g(\a')g(\a'\ol{\b\g})}{g^\0(\a'\ol\b)g^\0(\a'\ol\g)}, 
\end{align*}
where 
$$C=\frac{g^\0(\a^2\ol\b)g^\0(\a^2\ol\g)}{g(\a^2)g(\a^2\ol{\b\g})}.$$
By Theorem \ref{saal} and Theorem \ref{t2.2} (i), we obtain 
\begin{align*}
\wh f(\m)=& - \m(-1) \frac{(\a^2+\b+\g)_\m}{(\e+\a^2\ol\b+\a^2\ol\g)^\0_\m} 
-C
 -\m(-1)\FF{\a^2,\b,\g}{\a^2\ol\b,\a^2\ol\g}{1}.
\end{align*}
Hence 
$$f(\l)=\FF{\a^2,\b,\g}{\a^2\ol\b,\a^2\ol\g}{-\l}
-\d(1-\l) C-\d(1+\l)\FF{\a^2,\b,\g}{\a^2\ol\b,\a^2\ol\g}{1}$$
for any $\l\in\k^*$, and the theorem is proved. 
\end{proof}

\section{Product Formulas}

Here, we prove several finite analogues of product formulas known for complex hypergeometric functions, as listed in \cite{bailey2}. 
Note that formulas in Theorem \ref{p2.5} and Section \ref{s5} can be regarded as product formulas involving ${}_1F_0(\l)$ (see Corollary \ref{c2.0}). 

Recall Kummer's product formulas (cf. \cite[(2.01), (2.02)]{bailey2})
\begin{align*}
e^{-x}{}_1F_1\left({a \atop b};x\right) & = {}_1F_1\left({b-a \atop b};-x\right), 
\\
e^{-\frac{x}{2}}{}_1F_1\left({a \atop 2a};x\right) & = {}_0F_1\left({ \atop a+\frac{1}{2}};\frac{x^2}{16}\right). 
\end{align*}
Their finite analogues are the following (see Proposition \ref{p2.1} (ii)). 

\begin{thm}\ 
\begin{enumerate}
\item If $(\a,\e+\b)=0$, then for any $\l\in\k$, 
$$\psi(\l)\FF{\a}{\b}{\l}= \FF{\ol\a\b}{\b}{-\l}.$$
\item If $p$ is odd and $\a \ne \e$, then for any $\l\in\k$, 
$$\psi\left(\frac{\l}{2}\right)\FF{\a}{\a^2}{\l}= \FF{}{\a\phi}{\frac{\l^2}{16}}. $$
\end{enumerate}
\end{thm}

\begin{proof}(i) The case $\l=0$ is clear. Otherwise, by Theorem \ref{iteration} and Proposition \ref{p2.1}, 
\begin{align*}
- j(\a,\ol\a\b)\FF{\a}{\b}{\l}
&= \sum_{x\in\k} \ol\psi(\l x) \a(x)\ol\a\b(1-x)
\\&= \ol\psi(\l)
\sum_{x\in\k}\ol\psi(-\l+\l x) \a(x)\ol\a\b(1-x)
\\& = \ol\psi(\l) \sum_{y\in\k} \ol\psi(-\l y) \ol\a\b(y)\a(1-y)
\\& = - j(\a,\ol\a\b) \ol\psi(\l)  \FF{\ol\a\b}{\b}{-\l}.
\end{align*}

(ii) The case $\l=0$ is clear.  Let $f(\l)$ (resp. $g(\l)$) denote the left (resp. right) member of the formula, viewed as a function on $\k^*$. 
Since $f(\l)$ is even by (i), it suffices to show that $\wh f(\n^2)=\wh g(\n^2)$ for any $\n\in\ck$.   
First, 
\begin{align*}
\wh g(\n^2)
&= \frac{1}{1-q}\sum_{\m\in\ck} \frac{\ol\m(16)}{(\e+\a\phi)^\0_\m}  \sum_{\l\in\k^*} \m^2\ol\n^2(\l)
=- \ol\n^2(4) \sum_{\m^2=\n^2}  \frac{1}{(\e+\a\phi)^\0_\m}. 
\end{align*}
On the other hand, 
\begin{align*}
\wh f(\n^2)
&=\frac{1}{1-q} \sum_{\m\in\ck} \frac{(\a)_\m}{(\e+\a^2)^\0_\m} \sum_{\l \in\k^*} \psi\left(\frac{\l}{2}\right) \m\ol\n^2(\l)
\\&= -\frac{1}{1-q} \sum_{\m\in\ck} \frac{(\a)_\m}{(\e+\a^2)^\0_\m} g(\m\ol\n^2) \m\ol\n^2(2)
\\&=-\frac{1}{1-q}\ol\n(4) g(\ol\n^2) \sum_{\m\in\ck} \frac{(\a+\ol\n^2)_\m}{(\e+\a^2)^\0_\m}  \m(2)
\\&=-\ol\n(4) g(\ol\n^2) \FF{\a,\ol\n^2}{\a^2}{2}. 
\end{align*}
If $(\n^2,\e+\ol\a^2)=0$, then by Theorem \ref{connection}, Theorem \ref{t2.3} (ii) and the duplication formula, 
it becomes
$$-\ol\n(4)\sum_{\m^2=\n^2} \frac{qg^\0(\a^2) g(\a\m)}{g(\a)g(\a^2\n^2)g^\0(\m)}
=-\ol\n^2(4)\sum_{\m^2=\n^2}  \frac{qg^\0(\a\phi)}{g^\0(\m)g(\a\m\phi)}=\wh g(\n^2).$$
If $\n^2=\e$, then one verifies using Theorem \ref{reduction}, Proposition \ref{shift} and the duplication formula that
$$\wh f(\e)=\wh g(\e)=-\frac{qg^\0(\a\phi)}{g(\a)g(\phi)}-1.$$
If $\n^2=\ol\a^2 \ne \e$, then one verifies similarly that 
$$\wh f(\ol\a^2)=\wh g(\ol\a^2)=-\a^2(4)\left(\frac{qg(\a\phi)}{g(\ol\a)g(\phi)} +\frac{g(\a\phi)}{g(\ol\a\phi)}\right).$$
Hence the proof is complete. 
\end{proof}

Next, recall Ramanujan's formula (cf. \cite[(2.09)]{bailey2})  
$${}_1F_1\left({a\atop 2b};x\right) {}_1F_1\left({a\atop 2b};-x\right) 
={}_2F_3\left({a, 2b-a \atop 2b, b, b+\frac{1}{2}};\frac{x^2}{4}\right).$$
Its finite analogue is the following. 

\begin{thm}\label{t6.2}
If $p$ is odd and $(\a,\e+\b+\b\phi+\b^2)=0$, then 
$$\FF{\a}{\b^2}{\l}\FF{\a}{\b^2}{-\l} =\FF{\a,\ol\a\b^2}{\b^2,\b,\b\phi}{\frac{\l^2}{4}}.$$
\end{thm}

\begin{proof}
Let $f(\l)$ (resp. $g(\l)$) denote the left (resp. right) member. 
Since both $f$ and $g$ are even, it suffices to show that $\wh f(\n^2)=\wh g(\n^2)$ for any $\n\in\ck$. 
First, 
$$\wh g(\n^2)= -\sum_{\m^2=\n^2} \frac{(\a+\ol\a\b^2)_\m}{(\e+\b^2+\b+\b\phi)^\0_\m} \ol\m(4). $$
On the other hand, by the convolution formula (see \ref{s.fourier}),  
\begin{align*}
\wh f(\n^2)&=-\frac{1}{1-q} \sum_{\m\m'=\n^2} \frac{(\a)_\m}{(\e+\b^2)^\0_\m}\frac{(\a)_{\m'}}{(\e+\b^2)^\0_{\m'}}\m'(-1)
\\&=  -\frac{1}{1-q}\frac{(\a)_{\n^2}}{(\e+\b^2)^\0_{\n^2}} \sum_\m \frac{(\a)_\m}{(\e+\b^2)^\0_\m} \frac{(\ol{\n^2}+\ol{\b^2\n^2})_\m}{(\ol{\a\n^2})^\0_\m}
\\&=-\frac{(\a)_{\n^2}}{(\e+\b^2)^\0_{\n^2}} \FF{\a,\ol{\n^2},\ol{\b^2\n^2}}{\b^2,\ol{\a\n^2}}{1}, 
\end{align*}
where we used Lemma \ref{l2.1}. 
We can apply Theorem \ref{t2.2} (i), and obtain
$$ \FF{\a,\ol{\n^2},\ol{\b^2\n^2}}{\b^2,\ol{\a\n^2}}{1}
= \sum_{\m^2=\n^2} \frac{(\e)^\0_{\n^2}  (\a+\ol\a\b^2)_\m}{(\a)_{\n^2}(\e+\b^2)^\0_\m}.$$
Then $\wh f(\n^2)=\wh g(\n^2)$ follows by the duplication formula.  
\end{proof}

\begin{lem}\label{product}
Let $\baa$, $\bbb$, $\baa'$, $\bbb' \in P$, and put 
$$f(\l)=F(\baa,\bbb;\l), \quad g(\l)=F(\baa',\bbb';\l).$$ 
Then for any $\n\in\ck$, 
$$\wh{fg}(\n)=-\frac{(\baa')_\n}{(\bbb')^\0_\n} F\left(\baa+\ol{\bbb'\n},\bbb+\ol{\baa'\n};(-1)^ {\deg(\baa'+\bbb')}\right).$$ 
\end{lem}

\begin{proof}Similar to the proof of Theorem \ref{t6.2}. 
\end{proof}

We have Whipple's formula between two terminating Saalsch\"utzian ${}_4F_3(1)$'s (cf. \cite[7.2. (1)]{bailey})
$${}_4F_3\left({a,b,c,-n \atop e,f,g};1\right)=\frac{(f-c)_n(g-c)_n}{(f)_n(g)_n} {}_4F_3\left({e-a,e-b,c,-n \atop e,1+c-f-n,1+c-g-n};1\right). $$
The following is a finite analogue (cf. \cite[(5.12)]{greene}). 

\begin{thm}\label{t6.4}
Suppose that $\a\b\vp\psi=\g\s\t$ and $(\a+\b,\e+\g)=(\vp+\psi,\s+\t)=0$. Then 
\begin{align*}
\FF{\a,\b,\vp,\psi}{\g,\s,\t}{1}
=&\frac{(\s\ol\psi)_{\ol\vp}(\t\ol\psi)_{\ol\vp}}{(\s)^\0_{\ol\vp}(\t)^\0_{\ol\vp}} 
\FF{\ol\a\g,\ol\b\g,\vp,\psi}{\g,\ol\s\vp\psi,\ol\t\vp\psi}{1}
\\&+ q^{-\d(\a\b\ol\g)} \frac{g^\0(\g)g^\0(\s)g^\0(\t)}{g(\a)g(\b)g(\vp)g(\psi)}
\\&-q^{-\d(\a\b\ol\g)}\g\vp\psi(-1)  \frac{g(\a\ol\g)g(\b\ol\g)g^\0(\g)g^\0(\s)g^\0(\t)}{g(\vp)g(\psi)g(\s\ol\vp)g(\t\ol\vp)g(\s\ol\psi)g(\t\ol\psi)}.
\end{align*}
\end{thm}

\begin{proof}
Suppose that $\a\b\g'=\a'\b'\g$ and $(\a+\b,\e+\g)=(\a'+\b',\e+\g')=0$, 
and put functions on $\k^*$ as 
$$f(\l)=\FF{\a,\b}{\g}{\l} \FF{\ol{\a'}\g',\ol{\b'}{\g'}}{\g'}{\l}, \  
g(\l)=\FF{\ol\a\g,\ol\b\g}{\g}{\l}\FF{\a',\b'}{\g'}{\l}.$$
By Theorem \ref{p2.5} (i), we have 
$$f(\l)-\d(1-\l)f(1)=g(\l)-\d(1-\l) g(1).$$ 
Comparing the Fourier transforms using Lemma \ref{product}, we have for any $\n\in\ck$,  
\begin{align*}
&\frac{(\ol{\a'}\g')_\n(\ol{\b'}\g')_\n}{(\e)^\0_\n(\g')^\0_\n}
\FF{\a,\b,\ol\n,\ol{\g'\n}}{\g,\a'\ol{\g'\n},\b'\ol{\g'\n}}{1}
+\FF{\a,\b}{\g}{1} \FF{\ol{\a'}\g',\ol{\b'}{\g'}}{\g'}{1}
\\&=\frac{(\a')_\n(\b')_\n}{(\e)^\0_\n(\g')^\0_\n}\FF{\ol\a\g,\ol\b\g,\ol\n,\ol{\g'\n}}{\g,\ol{\a'\n},\ol{\b'\n}}{1}
+\FF{\ol\a\g,\ol\b\g}{\g}{1}\FF{\a',\b'}{\g'}{1}.
\end{align*}
Replacing $\n$, $\a'$, $\b'$, $\g'$ respectively with $\ol\vp$, $\s\ol\psi$, $\t\ol\psi$, $\vp\ol\psi$ and using Theorem \ref{summation}, we obtain the result. 
\end{proof}

Finally, recall Clausen's product formula (cf. \cite[(2.5.7)]{slater})
$$
{}_2F_1\left({a,b\atop a+b+\frac{1}{2}};x\right)^2
={}_3F_2\left({2a,2b,a+b \atop 2a+2b,a+b+\frac{1}{2}};x\right). 
$$
Its finite analogue is the following (cf. \cite[Theorem 1.5]{evans2}). 

\begin{thm}\label{clausen}
Suppose that $(\a^2+\b^2+\a\b,\e)=(\a,\b\phi)=0$. 
Then for any $\l\in\k^*$, 
\begin{align*}
&\FF{\a,\b}{\a\b\phi}{\l}^2 + \d(1-\l) \left(\frac{g^\0(\a\b\phi)g(\phi)}{g(\a)g(\b)}\right)^2
\\&=\FF{\a^2,\b^2,\a\b}{\a^2\b^2,\a\b\phi}{\l}
+q \frac{g^\0(\a\b\phi)}{g(\a)g(\b)}
\frac{g^\0(\a\b\phi)}{g(\a\phi)g(\b\phi)}
\a\b(\l^{-1}) \phi(1-\l^{-1}).
\end{align*}
In particular, 
$$\FF{\a^2,\b^2,\a\b}{\a^2\b^2,\a\b\phi}{1}=
\left(\frac{g^\0(\a\b\phi)g(\phi)}{g(\a)g(\b)}\right)^2+\left(\frac{g^\0(\a\b\phi)g(\phi)}{g(\a\phi)g(\b\phi)}\right)^2$$
and 
$$\FF{\a,\b}{\a\b\phi}{\l}^2=\FF{\a\phi,\b\phi}{\a\b\phi}{\l}^2 \quad (\l\ne 1).$$
\end{thm}

\begin{proof}
Put $f(\l) = \FF{\a,\b}{\a\b\phi}{\l}^2$. 
Since the Fourier transform of $\a\b(\l^{-1})\phi(1-\l^{-1})$ is 
$$-j(\a\b\n,\phi)=-\frac{g(\a\b)g(\phi)}{g^\0(\a\b\phi)} \frac{(\a\b)_\n}{(\a\b\phi)^\0_\n}$$
(Example \ref{e2.2} (iii)), we are reduced to prove 
\begin{equation}\label{*}
-\wh f(\n)=
\frac{(\a^2+\b^2+\a\b)_{\n}}{(\e+\a^2\b^2+\a\b\phi)^\0_\n}
+G_1 \frac{(\a\b)_\n}{(\a\b\phi)^\0_\n} + G_2 \tag{*}
\end{equation}
for any $\n\in\ck$, where we put 
$$G_1
= q\frac{g^\0(\a^2\b^2)}{g(\a^2)g(\b^2)}, \quad 
G_2=\left(\frac{g^\0(\a\b\phi)g(\phi)}{g(\a)g(\b)}\right)^2.$$
By Lemma \ref{product}, 
$$-\wh f(\n)= 
\frac{(\a+\b)_\n}{(\e+\a\b\phi)^\0_\n}
\FF{\a,\b,\ol{\a\b\phi\n}, \ol\n}{\a\b\phi, \ol{\a\n},\ol{\b\n}}{1}.$$

First, we prove the generic case where $\ol\n\not\in\{\a\b,\a\b\phi,\a^2,\a^2\b^2\}$. 
Note that $(\a^2)^\0_\n=(\a^2)_\n$ and $(\a\b)^\0_\n=(\a\b)_\n$ by assumption. 
In Theorem \ref{t6.4}, replace $\a$, $\b$, $\vp$, $\psi$, $\g$, $\s$, $\t$ respectively with 
$\b$, $\ol{\a\b\phi\n}$, $\ol\n$, $\a$, $\a\b\phi$, $\ol{\a\n}$, $\ol{\b\n}$.  
Then we have
$$-\wh f(\n)= A(\n)-B(\n)+ q^{-\d(\a^2\b\n)} G_2
$$
where we put
\begin{align*}
A(\n)&=\frac{(\a^2+\a\b)_\n}{(\e+\a\b\phi)^\0_\n} \FF{\a,\a\phi,\a^2\b^2\n,\ol\n}{\a^2,\a\b,\a\b\phi}{1}, 
\\B(\n)&= q^{-\d(\a^2\b\n)}
\frac{(\a^2+\a\b)_\n}{(\a^2\b^2+\a\b\phi)^\0_\n}\n(-1) \ol\b(4). 
\end{align*}
In Theorem \ref{t5.1} (i), replace $\a$, $\b$, $\g$, $\vp$ respectively with 
$\b\phi$, $\a^2\b^2\n$, $\ol\n$, $\a\b\phi$. Then 
\begin{align*}
A(\n) = 
&q^{-\d(\n)} \n(-1) \frac{(\a^2)_\n(\a\b)_\n}{(\a^2\b^2)^\0_\n(\a\b\phi)^\0_\n} C(\n) 
- \frac{(\a^2)_\n(\a\b)_\n}{(\e)^\0_\n(\a\b\phi)^\0_\n} D(\n), 
\end{align*}
where we put 
\begin{align*}
C(\n)& = \FF{\b^2,\a^2\b^2\n,\ol\n}{\a^2\b^2\n,\ol\n}{-1}, 
\\
D(\n)&=
\frac{1}{1-q} \n(-1) c(\a^2\b^2\n,\ol\n)\sum_{\m\in\{\ol{\a^2\b^2\n},\n\}} 
\frac{(\b^2)_\m}{(\e)^\0_\m}.
\end{align*}
By Theorem \ref{reduction}, we have
$$C(\n)=
q^{\d(\n)}\left(\ol\b(4)
+q^{-1} \n(-1)  \frac{(\b^2)_\n}{(\e)^\0_\n}
+q^{-1} \n(-1) G_1 \frac{(\a^2\b^2)^\0_\n}{(\a^2)_\n} \right).$$
On the other hand, 
$$D(\n)=
\frac{1-q}{q} \left(\frac{(\b^2)_\n}{(\a^2\b^2)^\0_\n}
+G_1 \frac{(\e)^\0_\n}{(\a^2)_\n}\right).$$
Now, the formula \eqref{*} follows immediately unless $\n=\ol{\a^2\b}$. 
The case $\n=\ol{\a^2\b}$ follows since 
$$\frac{(\a^2)_\n(\a\b)_\n}{(\a^2\b^2)^\0_\n(\a\b\phi)^\0_\n}\n(-1) \ol\b(4)
=G_2
$$
by the duplication formula. 

For the remaining cases, the formula \eqref{*} is verified using Theorem \ref{reduction},  Theorem \ref{summation} for $\n=\ol{\a\b}$, $\ol{\a\b\phi}$, and Theorem \ref{saal} for $\n=\ol{\a^2}$, $\ol{\a^2\b^2}$, together with the duplication formula and 
Proposition \ref{shift} for $\n=\ol{\a\b\phi}$. 
For example if $\n=\ol{\a^2\b^2}$, then 
\begin{align*}
-\wh{f}(\n) & = q^{-\d(\a\b\phi)} \frac{g(\phi)^2g^\0(\a\b\phi)^2 g(\a^2\b^2)}{g(\a)g(\b)g^\0(\a\b^2)g^\0(a^2\b)}
\FF{\a,\b,\a^2\b^2,\a\b\phi}{\a\b^2,\a^2\b,\a\b\phi}{1}
\\&=
\frac{g(\phi)^2g^\0(\a\b\phi)^2 g(\a^2\b^2)}{g(\a)g(\b)g^\0(\a\b^2)g^\0(a^2\b)}
\FF{\a,\b,\a^2\b^2}{\a\b^2,\a^2\b}{1} + 
\left( \frac{g(\phi)g^\0(\a\b\phi)g(\a^2\b^2)}{g(\a^2)g(\b^2)g(\a\b)}\right)^2 
\end{align*}
and
$$\FF{\a,\b,\a^2\b^2}{\a\b^2,\a^2\b}{1} 
=\frac{g(\a)g(\b)g^\0(\a\b^2)g^\0(\a^2\b)}{g(\a^2)g(\b^2)g(\a\b)^2}
+\frac{g^\0(\a\b^2)g^\0(\a^2\b)}{g(\a)g(\b)g(\a^2\b^2)},$$
hence the formula follows. 
\end{proof}

\section{Zeta functions of certain K3 surfaces} 

Recall that for a variety $X$ over $\k$, its zeta function is defined by the power series
$$Z(X,t)=\exp\left(\frac{\# X(\k_n)}{n} t^n\right),$$ 
where $\k_n$ is the extension of $\k$ of degree $n$ in a fixed algebraic closure. 
Here, we relate the zeta functions of certain K3 surfaces with those of elliptic curves.  

Let $p\ne 2$ and $E$ be an elliptic curve defined by 
$y^2=f(x)$, where $f(x) \in \k[x]$ is of degree $3$ with no multiple root. 
Then 
$$Z(E,t)=\frac{1-a(E)t+qt^2}{(1-t)(1-qt)}$$
where 
$$a(E):=1+q-\# E(\k) = -\sum_{x\in\k}\phi(f(x)). $$
By the Weil conjecture for elliptic curves proved by Hasse, 
$$1-a(E)t+qt^2=(1-\a t)(1-\ol\a t)$$
for some $\a\in \C$ with $|\a|=\sqrt q$ (cf. \cite[V, \S4]{silverman}). 

From now on, suppose that $4 \mid q-1$ and let $\s \in \ck$ be a quartic character, so that $\s^2=\phi$. 
\begin{ppn}\label{p7.1}
For $\l \in \k\setminus\{0,1\}$, let $E_\l$ be the elliptic curve over $\k$ defined by 
$$y^2=(1-x)(1-\l x^2).$$
Then 
$$a(E_\l)=\ol\s(-\l) \FF{\s,\s}{\e}{1-\l}.$$
\end{ppn}

\begin{proof}
Put a function on $\k^*$ as
$f(\l)=-\sum_{x\in\k^*}\phi((1-x)(1-\l x^2)).$
Then for any $\n\in\ck$, 
\begin{align*}
& \wh f(\n)= -\sum_{x\in\k^*} \phi(1-x)\n^2(x) \sum_{\l\in\k^*} \ol\n(\l x^2) \phi(1-\l x^2)
\\&= -j(\phi,\n^2)j(\ol\n,\phi)
=- \frac{(\e)_{\n^2} (\e)_{\ol\n}}{(\phi)^\0_{\n^2} (\phi)^\0_{\ol\n}} 
=- \frac{(\e)_{\n^2} (\phi)_{\n}}{(\phi)^\0_{\n^2} (\e)^\0_{\n}} 
=- \frac{(\e+2\phi)_{\n}}{(\e+\s+\s\phi)^\0_{\n}}. 
\end{align*}
Hence by Theorem \ref{reduction}, Proposition \ref{shift} and Theorem \ref{connection},  
\begin{align*}
f(\l)& = F(\e+2\phi,\e+\s+\s\phi;\l)
=q F(2\phi,\s+\s\phi;\l)+1
\\& =  \ol\s(-\l) \frac{g(\s)^2}{g(\phi)} \FF{\s,\s}{\phi}{\l}+1
= \ol\s(-\l) \FF{\s,\s}{\e}{1-\l}+1. 
\end{align*}
Since $a(E_\l)=f(\l)-1$, the proposition follows. 
\end{proof}

If $X$ is a K3 surface over $\k$, then by the Weil conjecture proved by Deligne, its zeta function is of the form
$$Z(X,t)=\frac{1}{(1-t)P(t)(1-q^2t)}, $$
where $P(t)$ is a polynomial of degree $22$ whose reciprocal roots have absolute value $q$. 

Now, for $\l \in \k\setminus\{0,1\}$, let $X_\l$ be the K3 surface defined by 
$$z^2=(1-\l xy)x(1-x)y(1-y).$$
Then we have by \cite[Proposition 4.1]{aop} 
$$\# X_\l(\k)= 1+ q^2 + 19 q+ b(\l),$$
where we put
$$b(\l)=\sum_{x, y \in \k} \phi((1-\l xy)x(1-x)y(1-y)).$$

We give a hypergeometric proof of the following theorem of Ahlgren--Ono--Penniston \cite[Theorem 1.1]{aop}, under the additional assumption that $4 \mid q-1$.  

\begin{thm}\label{t7.1} 
Let $\l\in\k\setminus\{0,1\}$ and 
$$1-a(E_{1-\l})t+qt^2=(1-\a t)(1-\ol\a t).$$ 
Then 
$$Z(X_\l,t)=\frac{1}{(1-t)(1-q^2t)(1-qt)^{19}(1-uqt)(1-u\a^2t)(1-u\ol\a^2t)},$$
where $u=\phi(1-\l)$. 
\end{thm}

\begin{proof}
By Corollary \ref{c2.1} (i) and Theorem \ref{clausen}, noting $j(\phi,\e)=1$ and $\phi(-1)=1$, 
$$b(\l)= \FF{\phi,\phi,\phi}{\e,\e}{\l}= \FF{\s,\s}{\e}{\l}^2 - q \phi(1-\l).$$
Hence by Proposition \ref{p7.1}, 
\begin{align*}
b(\l) = \phi(1-\l)(a(E_{1-\l})^2-q)=\phi(1-\l) (\a^2+\ol\a^2+q).
\end{align*}
If one replaces $\k$ with $\k_n$, then $\a$ is replaced with $\a^n$ and $u=\phi(1-\l)$ is replaced with 
$u^\frac{1-q^n}{1-q}=u^n$, 
so the theorem follows. 
\end{proof}

\begin{rmk}In \cite{aop}, the authors consider the elliptic curve 
$$E'_{1+\l} : (1+\l)y^2=(1-x)(1-(1+\l)x^2),$$
which is a quadratic twist of our $E_{1+\l}$, and a K3 surface isomorphic to our $X_{-\l}$. 
Since $a(E'_{1-\l})=\phi(1-\l) a(E_{1-\l})$, our statement agrees with theirs.  
\end{rmk}

Now, we consider the Dwork K3 surface $D_\l$ over $\k$ defined by the homogenous equation
$$x_1^4+x_2^4+x_3^4+x_4^4= 4\l x_1x_2x_3x_4 \quad (\l^4\ne 1).$$
By Nakagawa \cite{nakagawa}, its zeta function decomposes as 
$$Z(D_\l,t)=\frac{1}{(1-t)(1-q^2t)(1-qt)(1-uqt)^3(1-vqt)^3(1-wqt)^{12}P_\l(t)
},$$
where 
$u=\phi(1-\l^2)$, $v=\phi(1+\l^2)$, $w=uv\s(-1)$ and $P_\l(t)$ is a polynomial of degree $3$ such that 
$$P_\l(t)=\exp\left(\frac{F_n(\l)}{n}t^n \right), \quad 
F_n(\l)=\FF{\s_n,\s_n^2,\s_n^3}{\e_n,\e_n}{\l^{-4}}.$$
Here, $\e_n$ (resp. $\s_n$) is the trivial (resp. a quartic) character of $\k_n$. 

\begin{thm}Suppose that $1-\l^{-4} \in (\k^*)^2$. 
Let $\l'\in\k^*$ be 
a solution of $$1-\l^{-4}=\left(\frac{1+\l'}{1-\l'}\right)^2$$ and  let 
$$1-a(E_{1-\l'}) t + q t^2=(1-\a t)(1-\ol\a t).$$  
Then we have 
$$P_\l(t)=(1-qt)(1-\a^2t)(1-\ol\a^2t).$$
\end{thm}

\begin{proof}
By Theorem \ref{quad4} and the proof of Theorem \ref{t7.1}, we have 
$$\FF{\s,\s^2,\s^3}{\e,\e}{\l^{-4}}=\phi(1-\l') \FF{\phi,\phi,\phi}{\e,\e}{\l'}
=\a^2+\ol\a^2+q,$$
and the result follows similarly as before. 
\end{proof}

\begin{rmk}
In fact, $P_\l(t)$ is the characteristic polynomial of the Frobenius acting on the $3$-dimensional space mentioned in Remark \ref{r.dwork} (see \cite{nakagawa}).  
When $\k=\F_p$, Asakura \cite{asakura} obtained a more general result on $P_\l(t)$ 
without assuming $4 \mid p-1$ nor $1-\l^{-4} \in (\F_p^*)^2$, 
by studying the rigid cohomology of the family $D_\l$. 
Both $X_\l$ and $D_\l$ are known to be geometrically isogenous in the sense of Inose--Shioda to the Kummer surface associated to the self-product of  an elliptic curve. 
\end{rmk}

\appendix 

\section{A proof of the multiplication formula for Gauss sums}

Here we give an elementary proof of Theorem \ref{d-h}, using a geometric construction of Terasoma in his proof of the same theorem \cite[Theorem 3]{terasoma}.
While he uses $l$-adic cohomology, we count the number of rational points,  
the two objects being related by the Lefschetz trace formula. 

Let $X$ be a variety over $\k$ equipped with a left action of a finite group $G$, 
and suppose that the quotient variety $G\backslash X$ exists. 
Fix an algebraic closure $\ol\k$ of $\k$. 
Let $F$ denote the $q$th power Frobenius acting on $X=X(\ol\k)$. 
For each $g\in G$, put 
$$\L(X,g)=\#\{x \in X \mid Fx=gx\}.$$
For a character $\chi$ (of a $\C$-linear representation) of $G$, put 
$$N(X,\chi)= \frac{1}{\# G} \sum_{g\in G} \chi(g) \L(X,g).$$
We have the following functorialities. 

\begin{lem}[cf. {\cite[2.3]{serre}}]\label{la.1}\ 
\begin{enumerate}
\item
If $H \subset G$ is a normal subgroup and $\chi$ is a character of $G/H$, then 
$$N(X,\chi|_G) = N(H\backslash X, \chi).$$ 
\item
If $H \subset G$ is a subgroup and $\chi$ is a character of $H$, then 
$$N(X,\chi)= N(X, \Ind_H^G \chi).$$
\end{enumerate}
\end{lem}

Now we start the proof. 
Since the statement is obvious if $\a^n=\e$, we suppose that $\a^n\ne\e$. 
Then, by Proposition \ref{jacobi} (iv), we are reduced to prove 
$$\a^n(n) j(\underbrace{\a,\dots, \a}_\text{$n$ times}) = \prod_{\n^n=\e,\n\ne \e} j(\a,\n).$$

The following varieties, maps among them and group actions are all defined over $\k$. 
Let $X$ be a twisted Fermat hypersurface of dimension $n-1$ defined by 
$$t_1^{q-1} + \cdots + t_n^{q-1}= n, \quad t_1\cdots t_n \ne 0.$$
Let $C$ be a Fermat quotient curve defined by 
$$x^{q-1}+y^n=1, \quad x \ne 0.$$
Let $S\subset \A^n$ be a hyperplane defined by 
$$s_1+\cdots+s_n=n, \quad s_1\cdots s_n \ne 0,$$
and let $T\to S$ be a covering defined by $t^{q-1}=s_1\cdots s_n$. 
Define a map $X \to T$ by 
$$s_j=t_j^{q-1} \quad (j=1,2,\dots, n), \quad t=\prod_{j=1}^n t_j.$$
In this appendix, let $\mu_n$ denote the group of $n$th roots of unity in $\k$. 
Fix a primitive root $\z\in\m_n$ and 
define a map $C^{n-1} \to T$  by 
$$s_j=\prod_{i=1}^{n-1} (1-\z^j y_i) \quad (j=1,2,\dots, n), \quad
t=\prod_{i=1}^{n-1} x_i,
$$
where $(x_i,y_i)$ denotes the coordinates of the $i$th component of $C^{n-1}$. 
Then, $X$, $T$, $C^{n-1}$ are all Galois over $S$, and we have natural identifications 
$$\Gal(X/S)=\m_{q-1}^n, \quad \Gal(T/S)=\m_{q-1}, \quad 
\Gal(C^{n-1}/S)=\m_{q-1}^{n-1} \rtimes S_{n-1}.$$
The restriction maps 
$\Gal(X/S)\to\Gal(T/S)$ and $\Gal(C^{n-1}/S)\to\Gal(T/S)$
are identified respectively with the multiplication $\m_{q-1}^n\to\m_{q-1}$ and 
the first projection followed by the  multiplication $\m_{q-1}^{n-1}\to\m_{q-1}$. 

\begin{rmk}
Terasoma \cite{terasoma} also constructs a common covering of $X$ and $C^{n-1}$ over $S$. 
Let $C' \to C$ be a covering given by 
$u_j^{q-1}=1-\z^j y$ ($j=1,\dots, n$), $\prod_{j=1}^n u_j = x$. 
Then, as well as the map $C'^{n-1}\to C^{n-1}$,  
the map $C'^{n-1} \to X$ is given by $t_j= \prod_{i=1}^{n-1} u_{i,j}$ ($j=1,\dots, n$), and $C'^{n-1}$ is Galois over $S$ with $\Gal(C'^{n-1}/S)=(\m_{q-1}^n)^{n-1} \rtimes S_{n-1}$. 
\end{rmk}

For the given $\a \in \ck =\wh{\mu_{q-1}}$, put 
$\a^{(n)}=\a|_{\m_{q-1}^n}$ and $\chi=\a|_{\m_{q-1}^{n-1}\rtimes S_{n-1}}$.  
Then by Lemma \ref{la.1} (i), 
$$N(X,\a^{(n)}) =N(T,\a)= N(C^{n-1}, \chi).$$
By Weil \cite{weil} (cf. \cite[(2.12)]{koblitz}), we have 
$$N(X,\a^{(n)}) = (-1)^{n-1}\a^n(n) j(\underbrace{\a,\dots, \a}_\text{$n$ times}).$$

Our task is to compute $N(C^{n-1}, \chi)$. 
Let $C_0\subset C$ be the subvariety defined by $y\ne 0$, and put $D=C\setminus C_0$. 
For $k \in\N$, put 
$$H_k=\m_{q-1}^k \rtimes S_k, \quad G_k=(\m_{q-1}\times \m_n)^k \rtimes S_k$$ 
($S_0=\{1\}$ by convention). 
An element of $G_k$ is written as $\x\y\s$ with $\x=(\x_i)_{i=1,\dots, k} \in \m_{q-1}^k$, $\y=(\y_i)_{i=1,\dots, k}\in\m_n^k$ and $\s\in S_k$. 
Then $G_k$ acts naturally on $C^k$, respecting $C_0^k$ and $D^k$. 
The support of $\s\in S_k$ is defined by $\supp(\s)=\{i=1,\dots, k \mid \s(i)\ne i\}$.  
For $\x \in \m_{q-1}^k$ and $\s\in S_k$, put  
$p(\x)=\prod_{i=1}^k \x_i$, $p_\s(\x)=\prod_{i\in \supp(\s)} \x_i$, 
and similarly $p(\y)$, $p_\s(\y)$ for $\y\in\m_n^k$. 
Define a character $\chi_k \in \wh{H_k}$ by 
$$\chi_k(\x\s)=\a(p(\x)).$$ 
Note that $\chi=\chi_{n-1}$. 
For $\s\in S_k$, we write its cycle decomposition (unique up to ordering) as 
$\s=\s_1\cdots \s_r$, 
where $\s_j$ is a cyclic permutation of length $l_j$ with $\sum_{j=1}^r l_j=k$, and 
$\supp(\s_j)$'s are all disjoint.

\begin{lem}\label{la.2}
Let $\x\y\s\in G_k$ and  $\s=\s_1\cdots \s_r$ be the cycle decomposition. Then 
$$\Ind^{G_k}_{H_k} \chi_k (\x\y\s)= \prod_{j=1}^r \sum_{\vp \in \wh\m_n} \a(p_{\s_j}(\x))\vp(p_{\s_j}(\y)).$$
\end{lem}
\begin{proof} 
Since $\m_n^k$ represents $G_k/H_k$, we have by definition
$$\Ind^{G_k}_{H_k} \chi_k (\x\y\s)= \sum_{\y'\in\m_n^k, \y'^{-1}\y\s(\y')=1} \chi_k(\x).$$
The condition $\y'^{-1}\y\s(\y')=1$ is satisfied only when $p_{\s_j}(\y)=1$ for all $j$, 
and then the number of such $\y'$'s  is $n^r$. 
Therefore, for any $\y$, the number of $\y'$'s satisfying the condition is $\prod_{j=1}^r \sum_{\vp\in\wh{\m_n}} \vp(p_{\s_j}(\y))$. Since $\chi_k(\x)=\prod_j \a(p_{\s_j}(\x))$, the statement follows. 
\end{proof}

For an integer $l\ge 1$, let $\k_l$ denote the degree $l$ extension of $\k$ contained in $\ol\k$, and let 
$N_l\colon \k_l^* \to \k^*$ denote the norm map. 

\begin{lem}\label{la.3}
Let $\x\y\s\in G_k$ and  $\s=\s_1\cdots \s_r$ be the cycle decomposition. Then 
\begin{align*}
&\L(C_0^k,\x\y\s)= 
((q-1)n)^r 
\\&\times \prod_{j=1}^r \#\left\{(u,v) \in (\k_{l_j}^*)^2 \bigm| u+v=1, 
N_{l_j}(u)=p_{\s_j}(\x), N_{l_j}(v)^\frac{q-1}{n}=p_{\s_j}(\y) \right\}.
\end{align*}
\end{lem}
\begin{proof}
It reduces to the case $r=1$. Let $(x_i,y_i)_{i=1,\dots, k} \in C_0^k$. 
Then $\x\y\s (x_i,y_i)_i=F(x_i,y_i)_i$ happens only when 
$F^k(x_1,y_1)=(p(\x)x_1,p(\y)y_1)$, i.e. $x_1^{q^k-1}=p(\x)$, $y_1^{q^k-1}=p(\y)$. 
If we put $u=x_1^{q-1}$, $v=y_1^n$, then $u, v \in \k_{l_j}^*$, 
$u+v=1$ and the condition above becomes  
$N_k(u)=p(\x)$, $N_k(v)^\frac{q-1}{n}=p(\y)$. 
To each $(u,v)$ as above correspond $(q-1)n$ points $(x_1,y_1)$, hence the lemma.  
\end{proof}

\begin{ppn}\label{pa.1}We have
$$N(C_0^k,\chi_k)=\frac{(-1)^k}{k!} \sum_{\n_1,\dots, \n_k} 
\prod_{i=1}^k j(\a,\n_i),$$
where the sum is taken over all distinct $\n_1,\dots, \n_k\in\ck$ with $\n_1^n=\cdots=\n_k^n=\e$.  
\end{ppn}

\begin{proof}
By Lemma \ref{la.1} (ii), we have $N(C_0^k,\chi_k)=N(C_0^k, \Ind_{H_k}^{G_k} \chi_k)$. 
First, fix $\s=\s_1\cdots \s_r \in S_k$.  By Lemmas \ref{la.2} and \ref{la.3}, 
\begin{align*}
&\sum_{\x,\y} \Ind_{H_k}^{G_k} \chi_k(\x\y\s) \L(C_0^k,\x\y\s)
\\&=((q-1)n)^r \prod_j \sum_{\vp} ((q-1)n)^{l_j-1}
\sum_{u_j,v_j \in \k_{l_j}^*, u_j+v_j=1} \a(N_{l_j}(u))\vp (N_{l_j}(v)^\frac{q-1}{n}).
\end{align*}
Note that, for each $\x_0\in \m_{q-1}$, the number of $\x\in \m_{q-1}^{l_j}$ such that  
$p(\x)=\x_0$ is $(q-1)^{l_j-1}$, and similarly for $\y\in\m_n^{l_j}$. 
We identify $\vp\in\wh{\mu_n}$ with $\n\in\ck$ satisfying $\n^n=\e$ by 
$\n(v)=\vp(v^\frac{q-1}{n})$. Then, the last sum is written as $-j(\a\circ N_{l_j}, \n\circ N_{l_j})$, the Jacobi sum over $\k_{l_j}$. 
We have another well-known formula of Davenport--Hasse \cite{d-h} (cf. \cite[(5)]{weil})
$$j(\a\circ N_{l_j}, \n\circ N_{l_j})=j(\a,\n)^{l_j}.$$
Hence it follows 
$$\frac{1}{((q-1)n)^k} \sum_{\x,\y} \Ind_{H_k}^{G_k} \chi_k(\x\y\s) \L(C_0^k,\x\y\s)
= (-1)^r  \prod_{j=1}^r \sum_{\n \in \ck, \n^n=\e} j(\a,\n)^{l_j}.$$
Let us say that a $k$-tuple $(\n_1,\dots, \n_k)$ is $\s$-admissible if for each  $j=1,\dots, r$, $\n_i$'s agree  for all $i\in \supp(\s_j)$. 
Then the right-hand side is written as 
$$(-1)^r \sum_{(\n_1,\dots, \n_k): \text{$\s$-admissible}} \prod_{i=1}^k j(\a,\n_i).$$
Now we let $\s$ vary and write $r=r(\s)$. Then, 
$$N(C_0^k, \Ind_{H_k}^{G_k}\chi)
=\frac{1}{k!} \sum_{\n_1,\dots, \n_k}
\sum_{\s} (-1)^{r(\s)}\prod_{i=1}^k j(\a,\n_i),
$$
where the last sum is taken over $\s$ for which $(\n_1,\dots, \n_k)$ is $\s$-admissible. 
This sum vanishes unless $\n_1,\dots, \n_k$ are all distinct, since 
$$\sum_{\s \in S_l} (-1)^{r(\s)} = (-1)^l \sum_{\s \in S_l} \operatorname{sgn} \s= 0\quad (l \ge 2).$$
Hence the proposition is proved. 
\end{proof}

\begin{ppn}\label{pa.2}
For any $k\ge 0$, we have $N(D^k,\chi_k)=1.$
\end{ppn}

\begin{proof}
Since $D(\ol\k)=\{(x,0)\mid x\in\k^*\}$, it is fixed by $F$. 
For any $\x\s\in H_k$ with $\s=\s_1\cdots \s_r$ as before, 
$\L(D^k,\x\s)=(q-1)^{r(\s)}$ 
if $p_{\s_j}(\x)=1$ for all $j=1,\dots, r$, and $\L(D^k,\x\s)=0$ otherwise. 
The number of $\x$'s such that $p_{\s_j}(\x)=1$ for all $j$ is $\prod_j (q-1)^{l_j-1}$, 
and for such $\x$, we have $\chi(\x\s)=\prod_j \a(p_{\s_j}(\xi))=1$. 
Hence 
$N(D^k,\chi_k)=(\#H_k)^{-1} \sum_{\s\in S_k} (q-1)^k =1.$
\end{proof}

Now, let $(C^{n-1})_k \subset C^{n-1}$ denote the $S_{n-1}$-orbit of 
$D^k \times {C_0}^{n-1-k}$ ($k=0,\dots, n-1$). 
Then,  
\begin{align*}
N((C^{n-1})_k,\chi) &
= \frac{1}{\# H_{n-1}} \binom{n-1}{k} \sum_{\x_1\s_1\in H_k, \x_2\s_2\in H_{n-1-k}} 
\\
&\phantom{AAAAAAA}  \chi_k(\x_1) \chi_{n-1-k}(\x_2) \L(D^k,\x_1\s_1)\L(C_0^{n-1-k},\x_2\s_2) 
\\& = N(D^k,\chi_k) N(C_0^{n-1-k},\chi_{n-1-k}). 
\end{align*}
By Propositions \ref{pa.1} and \ref{pa.2}, noting $j(\a,\e)=1$, it follows 
$$N(C^{n-1},\chi)
=\sum_{k=0}^{n-1} N((C^{n-1})_k,\chi)
=(-1)^{n-1} \prod_{\n^n=\e, \n\ne\e} j(\a,\n).$$  
Hence the theorem is proved. 

\section*{Acknowledgements}
The author would like to thank Ryojun Ito and Akio Nakagawa for helpful comments. 
This work is supported by JSPS Grant-in-Aid for Scientific Research: 18K03234.


\end{document}